\documentclass[preprint,authoryear]{elsarticle}
\usepackage{amsmath}
\usepackage{amstext}
\usepackage{amsfonts}
\usepackage{amssymb}
\usepackage{graphicx}
\usepackage{bm}

\usepackage{tikz}
\usepackage{pgfplots}

\newcommand{\RR}{{\mathbb{R}}}
\newcommand{\NN}{{\mathbb{N}}}

\newcommand{\CC}{{\mathbb{C}}}

\newcommand{\trans}{{\sf T}}

\newcommand{\asto}{\overset{\rm a.s.}{\longrightarrow}}
\newcommand{\EE}{{\rm E}}

\DeclareMathOperator{\tr}{tr}
\DeclareMathOperator{\diag}{diag}

\newcounter{ctheorem}
\newtheorem{theorem}[ctheorem]{Theorem}
\newcounter{cassumption}
\newtheorem{assumption}[cassumption]{Assumption}

\newproof{proof}{Proof}

\newcounter{ccorollary}
\newtheorem{corollary}[ccorollary]{Corollary}
\newcounter{clemma}
\newtheorem{lemma}[clemma]{Lemma}

\newcounter{cremark}
\newtheorem{remark}[cremark]{Remark}

\journal{Journal of Multivariate Analysis}

\begin{document}

\begin{frontmatter}

\title{The Random Matrix Regime of Maronna's M-estimator with elliptically distributed samples\tnoteref{t1}}
\tnotetext[t1]{ Couillet's work is supported by the ERC MORE EC--120133. Silverstein is with Department of Mathematics, North Carolina State University, NC, USA. Silverstein's work is supported by the U.S. Army Research Office, Grant W911NF-09-1-0266.}



\author[supelec]{Romain Couillet}
\ead{romain.couillet@supelec.fr}
\address[supelec]{Telecommunication department, Sup\'elec, Gif sur Yvette, France}

\author[sondra]{Fr\'ed\'eric Pascal}
\ead{frederic.pascal@supelec.fr}
\address[sondra]{SONDRA Laboratory, Sup\'elec, Gif sur Yvette, France}

\author[ncsu]{Jack W. Silverstein}
\ead{jack@ncsu.edu}
\address[ncsu]{Department of Mathematics, North Carolina State University, NC, USA}




\begin{abstract}
	This article demonstrates that the robust scatter matrix estimator $\hat{C}_N\in\CC^{N\times N}$ of a multivariate elliptical population $x_1,\ldots,x_n\in\CC^N$ originally proposed by Maronna in 1976, and defined as the solution (when existent) of an implicit equation, behaves similar to a well-known random matrix model in the limiting regime where the population $N$ and sample $n$ sizes grow at the same speed. We show precisely that $\hat{C}_N\in\CC^{N\times N}$ is defined for all $n$ large with probability one and that, under some light hypotheses, $\Vert \hat{C}_N-\hat{S}_N\Vert\to 0$ almost surely in spectral norm, where $\hat{S}_N$ follows a classical random matrix model. As a corollary, the limiting eigenvalue distribution of $\hat{C}_N$ is derived. This analysis finds applications in the fields of statistical inference and signal processing. 
\end{abstract}	


\end{frontmatter}


\section{Introduction and problem statement}

The recent advances in the spectral analysis of large dimensional random matrices, and particularly of matrices of the sample covariance type, have triggered a new wave of interest for (sometimes old) problems in statistical inference and signal processing, which were usually treated under the assumption of a small population size versus a large sample dimension and are now explored assuming similar population and sample sizes. For instance, new source detection schemes have been proposed \citep{CAR08,BIA10,NAD10} based on the works on the extreme eigenvalues of large Wishart matrices \citep{TRA96,ELK07,BAI05,BAI08b,BAI06,COU11e}. New subspace methods in large array processing have also been derived \citep{MES08c,LOU10b,HAC11} that outperform the original MUSIC algorithm \citep{SCH86} by exploiting statistical inference methods on large random matrices \citep{MES08b,COU10b,LOU10}. Most of these signal processing methods fundamentally rely on the structure of the sample covariance matrix $\frac1n\sum_{i=1}^n x_ix_i^*$ formed from independent or linearly dependent (say zero mean) population vectors $x_1,\ldots,x_n\in\CC^N$, the asymptotic spectral properties of which are now well understood \citep{MAR67,SIL95,CHO95,BAI06,ELK07,MES08b}. The field of signal processing however covers a much wider scope than that of sample covariance matrices. Of interest here are the robust scatter matrix estimation techniques (a subclass of M-estimation \citep[Chapter~5]{VAN00}) that allow for a better -- more robust -- empirical approximation of population covariance (or scatter) matrices whenever (i) the probability distribution of the population vectors $x_i$ is heavy-tailed (or is at least far from Gaussian) or (ii) a small proportion of the samples $x_i$ presents an outlier behavior (i.e. follows an unknown distribution, quite different from the distribution of most samples) \citep{HUB64,HUB81,MAR06}. 

While classical sample covariance matrices exhibit a rather simple dependence structure (as they are merely the sum of independent or linearly dependent rank-one matrices), robust scatter matrix estimators are usually of a much more complex form which does not allow for standard random matrix analysis. This work specifically considers a widely spread model of robust scatter estimator, proposed in \citep{MAR76}, which contains as special cases the maximum-likelihood estimator of the scatter matrix for elliptically distributed population vectors, and which is well-behaved and mostly understood in the classical regime where $n\to\infty$ while $N$ is fixed. It is in particular shown in \citep{MAR76} that under some conditions the estimator is well-defined as the unique solution of a fixed-point equation and that the robust estimator converges almost surely (a.s.) to a deterministic matrix (which can be the scatter matrix for elliptical distribution of $x_i$ under correct parametrization). In this article, we revisit the study of Maronna's estimator for elliptically distributed samples using a probabilistic approach (as opposed to the statistical approach used classically in robust estimation theory) under the assumption that $n$ and $N$ are both large and of the same order of magnitude. This work follows after \citep{paper} where the simpler case of vector samples $x_i$ with independent entries was explored. The intuition for the proof of the main results follows in particular from the proof of the main theorem in \citep{paper}. 

Studying robust scatter estimators in the large random matrix regime, i.e. as $N,n$ grow large at the same speed, has important consequences in understanding many signal processing algorithms exploiting these estimators \citep{PAS08,PAS08b}. It also allows one to derive improved methods for source detection and parameter estimation as in \citep{CAR08,BIA10,MES08c,LOU10b,HAC11} for sample covariance matrix-based estimators. Adaptations (and improvements) of these results to robust estimation are currently under investigation.

\bigskip

Before discussing our main results, we first introduce the notations and assumptions taken in this article. We let $x_1,\ldots,x_n\in\CC^N$ be $n$ random vectors defined by $x_i=\sqrt{\tau_i} A_Ny_i$, where $\tau_1,\ldots,\tau_n\in \RR^+$ and $y_1,\ldots,y_n\in\CC^{\bar{N}}$ are random and $A_N\in\CC^{N\times \bar{N}}$ is deterministic. We denote $c_N\triangleq N/n$ and $\bar{c}_N\triangleq \bar{N}/N$ and shall consider the following growth regime.
\begin{assumption}
	\label{ass:cN}
	For each $N$, $c_N<1$, $\bar{c}_N\geq 1$ and
\begin{align*}
	c_- < \liminf_n c_N \leq \limsup_n c_N < c_+
\end{align*}
where $0<c_- < c_+<1$. 
\end{assumption}

The robust estimator under consideration in this article is Maronna's M-estimator $\hat{C}_N$ defined, when it exists, as a (possibly unique) solution to the equation in $Z\in\CC^{N\times N}$
\begin{align}
	\label{eq:def_hatCN}
	Z = \frac1n \sum_{i=1}^n u \left( \frac1N x_i^* Z^{-1} x_i \right) x_ix_i^*
\end{align}
where $u$ satisfies the following properties:
\begin{itemize}
	\item[(i)] $u:[0,\infty)\to (0,\infty)$ is nonnegative continuous and non-increasing
	\item[(ii)] $\phi:x\mapsto xu(x)$ is increasing and bounded with $\lim_{x\to\infty} \phi(x)\triangleq \phi_\infty>1$
	\item[(iii)] $\phi_\infty<c_+^{-1}$.
\end{itemize}

Note that (ii) is stronger than Maronna's original assumption \cite[Condition~(C) p.~53]{MAR76} as $\phi$ cannot be constant on any open interval. The assumption (iii) is also not classical in robust estimation but obviously compliant with the large $n$ assumption made in classical works (for which $c_+=0$). The importance of both assumptions will appear clearly in the proof of the main results.

The statistical hypotheses on $x_1,\ldots,x_n$ are detailed below.
\begin{assumption}
	\label{ass:xi} 
	The vectors $x_i=\sqrt{\tau_i} A_Ny_i$, $i\in\{1,\ldots,n\}$, satisfy the following hypotheses: 
\begin{enumerate}
	\item\label{item:nu} the (random) empirical measure $\nu_n=\frac1n\sum_{i=1}^n{\bm\delta}_{\tau_i}$ satisfies $\int x\nu_n(dx)\asto 1$
	\item\label{item:0} there exist $\varepsilon<1-\phi_\infty^{-1}<1-c_+$ and $m>0$ such that, for all large $n$ a.s. $\nu_n( [0,m) ) < \varepsilon$
	\item\label{item:CN} defining $C_N\triangleq A_NA_N^*$, $C_N\succ 0$ and $\limsup_N \Vert C_N \Vert < \infty$
	\item\label{item:y} $y_1,\ldots,y_n\in\CC^{\bar{N}}$ are independent unitarily invariant complex (or orthogonally invariant real) zero-mean vectors with, for each $i$, $\Vert y_i\Vert^2=\bar{N}$, and are independent of $\tau_1,\ldots,\tau_n$.
\end{enumerate}
\end{assumption}

Item~\ref{item:nu} is merely a normalization condition which, along with Item~\ref{item:CN}, ensures the proper scaling and asymptotic boundedness of the model parameters. Note in particular that Item \ref{item:nu} ensures a.s. tightness of $\{\nu_n\}_{n=1}^\infty$, i.e. for each $\varepsilon>0$, there exists $M>0$ such that, with probability one, $\nu_n([M,\infty))<\varepsilon$ for all $n$. Item~\ref{item:0} mainly ensures that no heavy mass of $\tau_i$ concentrates close to zero; this will ensure the existence of a solution to \eqref{eq:def_hatCN} and avoid technical problems when a solution to \eqref{eq:def_hatCN} exists (and is therefore invertible) but has many eigenvalues close to zero.

Note that Item~\ref{item:y} could be equivalently stated as $y_i=\sqrt{\bar N}\frac{\tilde{y}_i}{\Vert \tilde{y}_i\Vert}$ with $\tilde{y}_i\in\CC^{\bar N}$ standard complex Gaussian (or standard real Gaussian). This remark will be used throughout the proofs of the main results which rely in part on random matrix identities for matrices with independent entries.

All these conditions are met in particular if the $\tau_i$ are independent and identically distributed (i.i.d.) with common unit mean distribution $\nu$ (in which case $\int x\nu_n(dx)\asto 1$ by the strong law of large numbers) such that $\nu(\{0\})=0$. If in addition $N=\bar{N}$, then $x_1,\ldots,x_n$ are i.i.d. zero-mean complex (or real) elliptically distributed with full rank \citep[Theorem~3]{OLI12}. In particular, if $\tau_1$ is Rayleigh distributed, $x_1$ is complex zero mean Gaussian. If $1/\tau_1$ is chi-squared distributed, $x_1$ is instead zero mean complex Student distributed, etc. (see \citep{OLI12} for further discussions and recent results on elliptical distributions).

For simplicity of exposition, most of the article, and in particular the proofs of the main results, will assume the case of complex $x_i$; the results remain however valid in the case of real random variables.

\begin{assumption}
	\label{ass:tau}
	For each $a>b>0$, a.s.
	\begin{align*}
		\limsup_{t\to\infty} \frac{\limsup_n \nu_n( (t,\infty) ) }{\phi(at)-\phi(bt)}=0.
\end{align*}
\end{assumption}

Assumption~\ref{ass:tau} controls the relative speed of the tail of $\nu_n$ versus the flattening speed of $\phi(x)$ as $x\to\infty$. Practical examples satisfying Assumption~\ref{ass:tau} are:
\begin{itemize}
	\item There exists $M>0$ such that, for all $n$, $\max_{1\leq i\leq n}\tau_i<M$ a.s. In this case, $\nu_n( (t,\infty) )=0$ a.s. for $t>M$ while $\phi(at)-\phi(bt)\neq 0$ since $\phi$ is increasing.
	\item For $u(t)=(1+\alpha)/(\alpha+t)$ for some $\alpha>0$, it is easily seen that it is sufficient that $\limsup_n\nu_n( (t,\infty) )=o(1/t)$ a.s. for Assumption~\ref{ass:tau} to hold. In particular, if the $\tau_i$ are i.i.d. with distribution $\nu$, $\limsup_n \nu_n( (t,\infty) )=\nu( (t,\infty) )$ a.s. and, by Markov inequality, it suffices that $\int x^{1+\varepsilon}\nu(dx)<\infty$ for some $\varepsilon>0$.
\end{itemize}

The main contribution of this article is twofold: we first present a result on existence and uniqueness of $\hat{C}_N$ as a solution to \eqref{eq:def_hatCN} (Theorem~\ref{th:uniqueness}) and then study the limiting spectral behavior of $\hat{C}_N$ as $N,n\to\infty$ (Theorem~\ref{th:main}). With respect to existence and uniqueness, we recall that \cite[Theorem~1]{MAR76} ensures the existence and uniqueness of a solution to \eqref{eq:def_hatCN} under the statistical hypothesis that each $N$-subset of $x_1,\ldots,x_n$ spans $\CC^N$ and that $\phi_\infty>n/(n-N)$. While the first condition is met with probability one for continuous distributions of $x_i$, the second condition is restrictive under Assumption~\ref{ass:cN} as it imposes $\phi_\infty>1/(1-c_-)$ which brings a loss in robustness for $c_-$ close to one.\footnote{As commented in \citep{MAR76}, small values of $\phi_\infty$ induce increased robustness to the expense of accuracy.} Our first result is a probabilistic alternative to \cite[Theorem~1]{MAR76} which states that for all large $n$ a.s.,\footnote{As is common in random matrix theory, the probability space under consideration is that engendered by the growing sequences $\{x_1,\ldots,x_n\}_{n=1}^\infty$, with $N,n$ satisfying Assumption~\ref{ass:cN}, so that an event $E_n$ holds true ``for all large $n$ a.s.'' whenever, with probability one, there exists $n_0$ for which $E_n$ is true for all $n\geq n_0$, this $n_0$ possibly depending on the sequence.} \eqref{eq:def_hatCN} has a unique solution. This result uses the probability conditions on $x_1,\ldots,x_n$ and also uses $\phi_\infty<c_+^{-1}$ which, as opposed to \cite[Theorem~1]{MAR76}, enforces more robust estimators. The uniqueness part of the result also imposes that $\phi$ be strictly increasing, while \cite[Theorem~1]{MAR76} allows $\phi(x)=\phi_\infty$ for all large $x$. 

As for the large dimensional behavior of $\hat{C}_N$, in the fixed $N$ large $n$ regime and for i.i.d. $\tau_i$, it is of the form $\hat{C}_N\asto V_N$ where $V_N$ is the unique solution to $V_N=\EE[u(\frac1Nx_1^*V_N^{-1}x_1)x_1x_1^*]$ \cite[Theorem~5]{MAR76}. When the $x_i$ are i.i.d. elliptically distributed and $u$ is such that $\hat{C}_N$ is the maximum-likelihood estimator for $C_N$, then $V_N=C_N$, leading to a consistent estimator for $C_N$. In the random matrix regime of interest here, we show that $\hat{C}_N$ does not converge in any classical sense to a deterministic matrix but satisfies $\Vert \hat{C}_n-\hat{S}_N\Vert\asto 0$ in spectral norm, where $\hat{S}_N$ follows a random matrix model studied in \citep{ZHA09,PAU09,HAC13}. As such, the spectral behavior of $\hat{C}_N$ is easily analyzed from that of $\hat{S}_N$ for $N,n$ large.

In the next section, we introduce some new notations that simplify the analysis of $\hat{C}_N$ and provide an insight on the derivation of our main result, Theorem~\ref{th:main}.

\bigskip

{\it Generic notations:} We denote $\lambda_1(X)\leq\ldots\leq \lambda_N(X)$ the ordered eigenvalues of any Hermitian (or symmetric) matrix $X$. The superscript $(\cdot)^*$ designates transpose conjugate (if complex) for vectors or matrices. The norm $\Vert\cdot\Vert$ is the spectral norm for matrices and the Euclidean norm for vectors. The cardinality of a finite discrete set $\Omega$ is denoted $|\Omega|$. Almost sure convergence is written $\asto$. We use the set notation $\CC^+=\{z\in\CC,\Im[z]>0\}$. The Hermitian (or symmetric) matrix order relations are denoted $A\succeq B$ for $A-B$ nonnegative definite and $A\succ B$ for $A-B$ positive definite. The Dirac measure at point $x\in\RR$ is denoted ${\bm\delta}_x$.

\section{Preliminaries}
\label{sec:preliminaries}

First note from the expression of $\hat{C}_N$ as a (hypothetical) solution to \eqref{eq:def_hatCN} that we can assume $C_N=I_N$ by studying $C_N^{-\frac12}\hat{C}_NC_N^{-\frac12}$ in place of $\hat{C}_N$. Therefore, here and in all the major proofs in the article, without generality restriction, we place ourselves under the assumption $C_N=A_NA_N^*=I_N$.

Our objective is to prove that $\hat{C}_N$ is a well behaved solution of \eqref{eq:def_hatCN} (for all large $n$ a.s.) and to study the spectral properties of $\hat{C}_N$ as $N,n$ grow large. However, the structure of dependence between the rank-one matrices $u(\frac1Nx_i^*\hat{C}_N^{-1}x_i) x_ix_i^*$, $i=1,\ldots,n$, makes the large dimensional analysis of $\hat{C}_N$ via standard random matrix methods impossible (see e.g. \citep{PAS11,SIL06,AND10}) as these methods fundamentally rely on the independence (or simple dependence) of the structuring rank-one matrices. We propose here to show that, in the large $N,n$ regime, $\hat{C}_N$ behaves similar to a matrix $\hat{S}_N$ whose structure is more standard and easily analyzed through classical random matrix results. For this we first need to rewrite the fundamental equation \eqref{eq:def_hatCN} in order to exhibit a sufficiently ``weak'' dependence structure in the expression of $\hat{C}_N$. This rewriting is performed in Section~\ref{sec:rewriting} below. This being done, we then prove that some weakly dependent terms can be well approximated by independent ones in the large $N,n$ regime. Since the final result does not take an insightful form, we provide below in Section~\ref{sec:hint} a hint on how to obtain it intuitively.

\subsection{Rewriting \eqref{eq:def_hatCN}}
\label{sec:rewriting}

We need to introduce some new notations that will simplify the coming considerations. Write $x_i=\sqrt{\tau_i} A_Ny_i\triangleq \sqrt{\tau_i} z_i$ and recall that $C_N=I_N$ for the moment (in particular, $\Vert z_i\Vert$ is of order $\sqrt{N}$ for most $z_i$). If $\hat{C}_N$ is well-defined, we denote $\hat{C}_{(i)}\triangleq \hat{C}_N- \frac1n u(\frac1Nx_i^*\hat{C}_N^{-1}x_i)x_ix_i^*$. 

Remark that $\hat{C}_{(i)}$ depends on $x_i$ only through the terms $u(\frac1Nx_j^*\hat{C}_N^{-1}x_j)$, $j\neq i$, in which the term $\hat{C}_N$ is built on $x_i$. But since $x_i$ is only one among a growing number $n$ of $x_j$ vectors, this dependence structure looks intuitively ``weak''. This informal weak dependence between $x_i$ and $\hat{C}_{(i)}$, along with classical random matrix theory considerations, suggests that the quadratic forms $\frac1N z_i^*\hat{C}_{(i)}^{-1}z_i$, $i=1,\ldots,n$, are all well approximated by $\frac1N\tr \hat{C}_N^{-1}$ (more precisely, this would roughly be a consequence of Lemma~\ref{le:trace_lemma} and Lemma~\ref{le:rank1perturbation} in the Appendix if $z_i$ and $\hat{C}_{(i)}$ were truly independent).

With this in mind, let us rewrite $\hat{C}_N$ as a function of $\frac1N z_i^*\hat{C}_{(i)}^{-1}z_i$ instead of $\frac1N x_i^*\hat{C}_N^{-1}x_i$, $i=1,\ldots,n$. For this, let $Z\in\CC^{N\times N}$ be positive definite such that for each $i$, $Z_{(i)}\triangleq Z-\frac1n u(\tau_i\frac1Nz_i^*Z^{-1}z_i)\tau_i z_iz_i^*$ is positive definite. Using the identity $(A+\tau zz^*)^{-1}z=A^{-1}z/(1+\tau z^*A^{-1}z)$ for invertible $A$, vector $z$, and positive scalar $\tau$, observe that 
\begin{align*}
	\frac1N z_i^* Z^{-1} z_i = \frac{\frac1N z_i^* Z_{(i)}^{-1}z_i }{1+ \tau_i u\left(\tau_i \frac1Nz_i^*Z^{-1}z_i \right) \frac1n z_i^*Z_{(i)}^{-1}z_i}.
\end{align*}
Hence,
\begin{align*}
	\frac1N z_i^* Z_{(i)}^{-1}z_i \left( 1 - c_N \tau_i u\left(\tau_i \frac1Nz_i^*Z^{-1}z_i \right) \frac1N z_i^* Z^{-1} z_i \right) = \frac1N z_i^* Z^{-1} z_i
\end{align*}
which, by the definition of $\phi$, is
\begin{align*}
	\frac1N z_i^* Z_{(i)}^{-1}z_i \left( 1 - c_N \phi\left(\tau_i \frac1Nz_i^*Z^{-1}z_i \right) \right) = \frac1N z_i^* Z^{-1} z_i.
\end{align*}
Using Assumption~\ref{ass:cN} and $\phi_\infty<c_+^{-1}$, taking $n$ large enough to have $\phi(x)\leq \phi_\infty<1/c_N$, this can be rewritten
\begin{align}
	\label{eq:obtaining_g}
	\frac1N z_i^* Z_{(i)}^{-1}z_i = \frac{\frac1N z_i^* Z^{-1} z_i}{1 - c_N \phi\left(\tau_i \frac1Nz_i^*Z^{-1}z_i \right)}.
\end{align}

Now, since $\phi$ is increasing, $g:[0,\infty)\to [0,\infty)$, $x\mapsto x/(1-c_N\phi(x))$ is increasing, nonnegative, and maps $[0,\infty)$ onto $[0,\infty)$. Thus, $g$ is invertible with inverse denoted $g^{-1}$. In particular, from \eqref{eq:obtaining_g},
	\begin{equation*}
		\tau_i\frac1N z_i^* Z^{-1} z_i = g^{-1}\left( \tau_i \frac1N z_i^* Z_{(i)}^{-1}z_i \right).
	\end{equation*}
Call now $v: [0,\infty)\to [0,\infty)$, $x\mapsto u\circ g^{-1}$. Since $g$ is increasing and nonnegative and $u$ is non-increasing, $v$ is non-increasing and positive. Moreover, $\psi:x\mapsto xv(x)$ satisfies:
\begin{align*}
	\psi(x) &= x u(g^{-1}(x)) = g(g^{-1}(x)) u(g^{-1}(x)) = \frac{\phi(g^{-1}(x))}{1-c_N\phi(g^{-1}(x))}
\end{align*}
which is increasing, nonnegative, and has limit $\psi^N_\infty \triangleq \phi_\infty/(1-c_N\phi_\infty)$ as $x\to\infty$. Hence, $v$ and $\psi$ keep the same properties as $u$ and $\phi$, respectively. 
	
With these notations, to prove the existence and uniqueness of a solution to \eqref{eq:def_hatCN}, it is equivalent to prove that the equation in $Z$
\begin{align*}
	Z &= \frac1n\sum_{i=1}^n \tau_i v\left( \tau_i \frac1N z_i^* Z_{(i)}^{-1} z_i \right)  z_iz_i^*
\end{align*}
has a unique positive definite solution. But for this, it is sufficient to prove the uniqueness of $d_1,\ldots,d_n\geq 0$ satisfying the $n$ equations:
\begin{align}
	\label{eq:def_hatCN_2}
	d_j = \frac1N z_j^*\left(\frac1n \sum_{i\neq j} \tau_i v \left( \tau_i d_i \right) z_iz_i^*\right)^{-1}z_j,~1\leq j\leq n.
\end{align}

Indeed, if these $d_i$ are uniquely defined, then so is the matrix 
\begin{align}
	\label{eq:hatCN_2}
	\hat{C}_N = \frac1n \sum_{i=1}^n \tau_i v \left( \tau_i d_i \right) z_iz_i^*
\end{align}
with $d_i=\frac1N z_i^* \hat{C}_{(i)}^{-1}z_i$, $\hat{C}_{(i)}=\hat{C}_N - \frac1n u(\frac1Nx_i^*\hat{C}_N^{-1}x_i)x_ix_i^*$ (the existence follows from taking the $d_i$ solution to \eqref{eq:def_hatCN_2} and write $\hat{C}_N$ as in \eqref{eq:hatCN_2}, while uniqueness follows from the fact that \eqref{eq:hatCN_2} cannot be written with a different set of $d_i$ from the uniqueness of the solution to \eqref{eq:def_hatCN_2}).

This is the approach that is pursued to prove Theorem~\ref{th:uniqueness}, based on the results from \cite{YAT95}. Equation~\eqref{eq:hatCN_2}, which is equivalent to \eqref{eq:def_hatCN} (with $\hat{C}_N$ in place of $Z$), will be preferably used in the remainder of the article.

\subsection{Hint on the main result}
\label{sec:hint}

Assume here that the $d_i$ above are indeed unique for all large $n$ so that $\hat{C}_N$ is well defined. We provide some intuition on the main result.

From the discussion in Section~\ref{sec:rewriting}, we may expect the terms $d_i$ to be all close to $\frac1N\tr \hat{C}_N^{-1}$ for $N,n$ large enough. We may also expect $\frac1N\tr \hat{C}_N^{-1}$ to have a deterministic equivalent $\gamma_N$, i.e. there should exist a deterministic sequence $\{\gamma_N\}_{N=1}^\infty$ such that $|\frac1N\tr \hat{C}_N^{-1}-\gamma_N|\asto 0$. Let us say that all this is true. Since $\frac1N\tr \hat{C}_N^{-1}$ is the Stieltjes transform $\frac1N\tr (\hat{C}_N-zI_N)^{-1}$ of the empirical spectral distribution of $\hat{C}_N$ at point $z=0$, and since $\hat{C}_N$ is expected to be close to $\frac1n\sum_i \tau_i v(\tau_i \gamma_N) z_iz_i^*$ with now $v(\tau_i \gamma_N)$ independent of $z_1,\ldots,z_n$, from classical random matrix works, e.g. \citep{SIL95}, we would expect that one such $\gamma_N$ be given by (recall that $C_N=I_N$)
\begin{align*}
	\gamma_N = \left( \frac1n \sum_{i=1}^n \frac{\tau_i v(\tau_i \gamma_N) }{1+c_N\tau_i v(\tau_i \gamma_N) \gamma_N} \right)^{-1}
\end{align*}
if this fixed-point equation makes sense at all. This can be equivalently written as
\begin{align}
	\label{eq:fp_gamma}
	1 = \frac1n\sum_{i=1}^n \frac{\psi(\tau_i \gamma_N)}{1+c_N\psi(\tau_i\gamma_N)}.
\end{align}

We in fact prove in Section~\ref{sec:main_results} that such a positive $\gamma_N$ is well defined, unique, and satisfies $\max_{1\leq i\leq n}|d_i-\gamma_N|\asto 0$ (under correct assumptions). Proving this result is the main difficulty of the article.

This convergence, along with \eqref{eq:hatCN_2}, will then ensure that for all large $n$ a.s.
\begin{align*}
	\left\Vert \hat{C}_N - \hat{S}_N \right\Vert \asto 0
\end{align*}
where 
\begin{align*}
	\hat{S}_N=\frac1n\sum_{i=1}^n v\left( \tau_i \gamma_N \right) \tau_i z_iz_i^*
\end{align*}
with $\gamma_N$ the unique positive solution to \eqref{eq:fp_gamma}. It is then immediate under Assumption~\ref{ass:xi}--\ref{item:CN} to see that the result holds true also for $C_N\neq I_N$.

The major interest of this convergence in spectral norm is that $\hat{S}_N$ is a known and easily manipulable object, as opposed to $\hat{C}_N$. The result therefore conveys a lot of information about $\hat{C}_N$ among which the fact that its largest and smallest eigenvalues are almost surely bounded and bounded away from zero for all large $n$ (which is not in general the case of $\frac1n\sum_{i=1}^nx_ix_i^*$ for $\tau_i$ with unbounded support).

\section{Main results}
\label{sec:main_results}

We now make the statements of Section~\ref{sec:hint} rigorous. The first result ensures the existence and uniqueness of a solution $\hat{C}_N$ to \eqref{eq:def_hatCN} for $n$ large enough.

\begin{theorem}[Uniqueness]
	\label{th:uniqueness}
	Let Assumptions~\ref{ass:cN}~and~\ref{ass:xi} hold, with $\limsup_N \Vert C_N\Vert$ non necessarily bounded. Then, for all large $n$ a.s., \eqref{eq:def_hatCN} has a unique solution $\hat{C}_N$ given by
	\begin{equation*}
		\hat{C}_N = \lim_{ \substack{t\to\infty \\ t\in \NN } } Z^{(t)}
	\end{equation*}
	where $Z^{(0)}\succ 0$ is arbitrary and, for $t\in \NN$,
	\begin{equation*}
		Z^{(t+1)} = \frac1n \sum_{i=1}^n u \left( \frac1Nx_i^*\left(Z^{(t)}\right)^{-1}x_i \right) x_ix_i^*.
	\end{equation*}
\end{theorem}

Having defined $\hat{C}_N$, the main result of the article provides a random matrix equivalent to $\hat{C}_N$, much easier to study than $\hat{C}_N$ itself.

\begin{theorem}[Asymptotic Behavior]
	\label{th:main}
	Let Assumptions~\ref{ass:cN}--\ref{ass:tau} hold, and let $\hat{C}_N$ be given by Theorem~\ref{th:uniqueness} when uniquely defined as the solution of \eqref{eq:def_hatCN} or chosen arbitrarily if not. Then
	\begin{align*}
		\left\Vert \hat{C}_N - \hat{S}_N \right\Vert \asto 0
	\end{align*}
	where
	\begin{align*}
		\hat{S}_N \triangleq \frac1n\sum_{i=1}^n v(\tau_i \gamma_N) x_ix_i^*
	\end{align*}
	and $\gamma_N$ is the unique positive solution of the equation in $\gamma$
	\begin{align*}
		1 = \frac1n\sum_{i=1}^n \frac{\psi(\tau_i\gamma)}{1+c_N\psi(\tau_i\gamma)}
	\end{align*}
	with the functions $v: x\mapsto (u\circ g^{-1})(x)$, $\psi: x\mapsto xv(x)$, and $g:\RR^+\to \RR^+,~x\mapsto x/(1-c_N\phi(x))$.
\end{theorem}

The fact that $\hat{C}_N$ is well approximated by $\hat{S}_N$, which follows a random matrix model studied extensively in \citep{PAU09,HAC13}, has important consequences. From a purely mathematical standpoint, this provides a full characterization of the spectral behavior of $\hat{C}_N$ for large $N,n$ (see in particular Corollary~\ref{cor:spectrum} below). For application purposes, this first enables the performance analysis in the large $N,n$ horizon of standard signal processing methods already relying on $\hat{C}_N$ (these methods were so far analyzed solely in the fixed $N$ large $n$ regime). A second, more important, consequence for signal processing application is the possibility to fully exploit the structure of $\hat{C}_N$ for large $N,n$ to improve existing robust schemes. Deriving such improved methods is not the subject of the current article but should be directly accessible from Theorem~\ref{th:main}, while performance analysis of these methods may demand supplementary treatment, such as central limit theorems for functionals of $\hat{C}_N$.

\begin{corollary}[Spectrum]
	\label{cor:spectrum}
	Let Assumptions~\ref{ass:cN}--\ref{ass:tau} hold. Then
	\begin{align}
		\label{eq:spectrum_conv}
		\frac1n \sum_{i=1}^n {\bm \delta}_{\lambda_i(\hat{C}_N)} - \mu_N \asto 0
	\end{align}
	where the convergence is in the weak probability measure sense, with $\mu_N$ a probability measure with continuous density and Stieltjes transform $m_N(z)$ given, for $z\in\CC^+$, by 
\begin{align*}
	m_N(z) = -\frac1{z}\frac1N \sum_{i=1}^N \frac{1}{1+\tilde{\delta}_N(z) \lambda_i(C_N) }
\end{align*}
where $\tilde{\delta}_N(z)$ is the unique solution in $\CC^+$ of the equations in $\tilde{\delta}$
\begin{align*}
	\tilde{\delta} &= -\frac1z\frac1n \sum_{i=1}^n \frac{\psi(\tau_i\gamma_N) }{\gamma_N + \psi(\tau_i\gamma_N) \delta} \\
	\delta &= -\frac1z \frac1n \sum_{i=1}^N \frac{\lambda_i(C_N)}{1+ \lambda_i(C_N) \tilde{\delta}} 
\end{align*}
and where $\gamma_N$ is defined in Theorem~\ref{th:main}. Besides, the support $\mathcal S_N$ of $\mu_N$ is uniformly bounded. 
If $C_N=I_N$, $m_N(z)$ is the unique solution in $\CC^+$ of the equation in $m$
\begin{align*}
	m = \left( -z + \gamma_N^{-1} \frac1n \sum_{i=1}^n \frac{\psi(\tau_i\gamma_N)}{1+c\gamma_N^{-1}\psi(\tau_i\gamma_N)m} \right)^{-1}.
\end{align*}

Also, for each $N_0\in \NN$ and each closed set $\mathcal A\subset \RR$ with $\mathcal A\cap \left(\bigcup_{N\geq N_0} \mathcal S_N\right)=\emptyset$,
	\begin{align}
		\label{eq:eigwise_in_support}
		\left| \left\{\lambda_i(\hat{C}_N)\right\}_{i=1}^N \cap \mathcal A \right| \asto 0
	\end{align}
	so that, in particular, 
	\begin{align}
		\label{eq:bounded_eigs}
		\limsup_N \Vert \hat{C}_N\Vert <\infty.
	\end{align}
\end{corollary}
\begin{proof}
	Equation~\eqref{eq:spectrum_conv} is obtained from the results of \citep{ZHA09} with notations similar to \citep{HAC13}. The characterization of $\mu_N$ follows from \citep{HAC13}, where more information can be found.
	The uniform boundedness of the support is a consequence of the boundedness of $\psi$ and $\gamma_N$, Lemma~\ref{le:gamma_N} in Section~\ref{sec:proofs}. Finally, the results \eqref{eq:eigwise_in_support} and \eqref{eq:bounded_eigs} are an application of \citep{PAU09} along with $\limsup_N \Vert \hat{S}_N\Vert\leq v(0) \limsup_N\Vert C_N\Vert \limsup_N \Vert \frac1n\sum_{i=1}^nx_ix_i^*\Vert<\infty$ by Assumption~\ref{ass:xi}--\ref{item:CN} and \citep{SIL98}.
\end{proof}

A consequence of Theorem~\ref{th:main} and Corollary~\ref{cor:spectrum} in the i.i.d. elliptical case is as follows.
\begin{corollary}[Elliptical case]
	Let Assumptions~\ref{ass:cN}--\ref{ass:tau} hold and in addition, let $\tau_i$ be i.i.d. with law $\nu$ and let $c_N\to c$. Then
	\begin{align*}
		\left\Vert \hat{C}_N - \frac1n\sum_{i=1}^n v(\tau_i \gamma^\infty) x_ix_i^* \right\Vert \asto 0
	\end{align*}
	where $\gamma^\infty$ is the unique positive solution to the equation in $\gamma$
	\begin{align*} 
		1 = \int \frac{\psi_c(t\gamma)}{1+c\psi_c(t\gamma)} \nu(dt)
	\end{align*}
	with $\psi_c=\lim_{c_N\to c}\psi$.
	Moreover, if $\frac1n\sum_{i=1}^n{\bm\delta}_{\lambda_i(C_N)}\to \nu^C$ weakly, then 
	\begin{align*}	
		\frac1n\sum_{i=1}^n{\bm\delta}_{\lambda_i(\hat{C}_N)}\asto \mu
	\end{align*}	
	weakly with $\mu$ a probability measure with continuous density of bounded support $\mathcal S$, the Stieltjes transform $m(z)$ of which is given for $z\in\CC^+$ by
	\begin{align*}
		m(z) = -\frac1z \int \frac1{1+\tilde{\delta}(z)t}\nu^C(dt)
	 \end{align*}
	 where $\tilde{\delta}(z)$ is the unique solution in $\CC^+$ of the equations in $\tilde{\delta}$
	 \begin{align*}
		 \tilde{\delta} &= -\frac1z\int \frac{\psi_c(t\gamma^\infty)}{\gamma^\infty+\psi_c(t\gamma^\infty)\delta}\nu(dt) \\
		 \delta &= -\frac{c}z\int \frac{t}{1+t\tilde{\delta}} \nu^C(dt).
	 \end{align*}
	 Finally, for every closed set $\mathcal A\subset \RR$ with $\mathcal A\cap \mathcal S=\emptyset$,
	 \begin{align*}
		 \left| \left\{\lambda_i(\hat{C}_N)\right\}_{i=1}^N \cap \mathcal A \right| \asto 0.
	 \end{align*}
\end{corollary}
\begin{proof}
	We use the fact that $\gamma_N\asto \gamma^\infty$ ($\gamma_N$ defined in Theorem~\ref{th:main}) which is a consequence of $\psi/(1+c_N\psi)$ being monotonous and $\gamma_N$ uniformly bounded, Lemma~\ref{le:gamma_N}. The rest unfolds from classical random matrix techniques.
\end{proof}

\begin{figure}[h!]
  \centering
  \begin{tikzpicture}[font=\footnotesize]
    \renewcommand{\axisdefaulttryminticks}{4} 
    \tikzstyle{every major grid}+=[style=densely dashed]       
    \tikzstyle{every axis y label}+=[yshift=-10pt] 
    \tikzstyle{every axis x label}+=[yshift=5pt]
    \tikzstyle{every axis legend}+=[cells={anchor=west},fill=white,
        at={(0.98,0.98)}, anchor=north east, font=\scriptsize ]
    \begin{axis}[
      xmin=0,
      ymin=0,
      xmax=2.3,
      ymax=3.5,
      bar width=1.5pt,
      grid=major,
      ymajorgrids=false,
      scaled ticks=true,
      xlabel={Eigenvalues},
      ylabel={Density}
      ]
      \addplot+[ybar,mark=none,color=black,fill=gray!40!white] coordinates{
	      (0.01,0.)(0.03,0.)(0.05,0.5)(0.07,3.1)(0.09,3.4)(0.11,2.9)(0.13,1.9)(0.15,0.7)(0.17,0.)(0.19,0.3)(0.21,0.8)(0.23,1.)(0.25,1.1)(0.27,1.2)(0.29,1.)(0.31,1.1)(0.33,1.2)(0.35,0.9)(0.37,0.9)(0.39,0.7)(0.41,0.9)(0.43,0.5)(0.45,0.4)(0.47,0.4)(0.49,0.1)(0.51,0.)(0.53,0.)(0.55,0.)(0.57,0.)(0.59,0.)(0.61,0.2)(0.63,0.3)(0.65,0.3)(0.67,0.2)(0.69,0.5)(0.71,0.4)(0.73,0.4)(0.75,0.4)(0.77,0.3)(0.79,0.6)(0.81,0.4)(0.83,0.4)(0.85,0.4)(0.87,0.5)(0.89,0.5)(0.91,0.4)(0.93,0.4)(0.95,0.5)(0.97,0.4)(0.99,0.5)(1.01,0.4)(1.03,0.4)(1.05,0.5)(1.07,0.4)(1.09,0.4)(1.11,0.4)(1.13,0.4)(1.15,0.4)(1.17,0.5)(1.19,0.4)(1.21,0.3)(1.23,0.4)(1.25,0.4)(1.27,0.4)(1.29,0.4)(1.31,0.4)(1.33,0.3)(1.35,0.4)(1.37,0.4)(1.39,0.3)(1.41,0.5)(1.43,0.3)(1.45,0.3)(1.47,0.3)(1.49,0.3)(1.51,0.3)(1.53,0.3)(1.55,0.4)(1.57,0.3)(1.59,0.4)(1.61,0.3)(1.63,0.2)(1.65,0.3)(1.67,0.2)(1.69,0.3)(1.71,0.2)(1.73,0.3)(1.75,0.2)(1.77,0.3)(1.79,0.3)(1.81,0.1)(1.83,0.2)(1.85,0.3)(1.87,0.2)(1.89,0.2)(1.91,0.2)(1.93,0.2)(1.95,0.2)(1.97,0.1)(1.99,0.2)(2.01,0.3)(2.03,0.1)(2.05,0.1)(2.07,0.1)(2.09,0.1)(2.11,0.)(2.13,0.3)(2.15,0.)(2.17,0.1)(2.19,0.)(2.21,0.1)(2.23,0.)
      };
      \addplot[black,line width=1pt] plot coordinates{
(0.038,0.004341)(0.04,0.004936)(0.042,0.005697)(0.044,0.006703)(0.046,0.008128)(0.048,0.010218)(0.05,0.013754)(0.052,0.021054)(0.054,0.050016)(0.056,1.334095)(0.058,1.966732)(0.06,2.377708)(0.062,2.648892)(0.064,2.874291)(0.066,3.035344)(0.068,3.153864)(0.07,3.242067)(0.072,3.315482)(0.074,3.373319)(0.076,3.404909)(0.078,3.422671)(0.08,3.430764)(0.082,3.424367)(0.084,3.416687)(0.086,3.399365)(0.088,3.37194)(0.09,3.344324)(0.092,3.310287)(0.094,3.266717)(0.096,3.225626)(0.098,3.18039)(0.1,3.130975)(0.102,3.077707)(0.104,3.01793)(0.106,2.960965)(0.108,2.901403)(0.11,2.838894)(0.112,2.773256)(0.114,2.704535)(0.116,2.632925)(0.118,2.558654)(0.12,2.481919)(0.122,2.402873)(0.124,2.321645)(0.126,2.238363)(0.128,2.153477)(0.13,2.062702)(0.132,1.969039)(0.134,1.872718)(0.136,1.768689)(0.138,1.661782)(0.14,1.545184)(0.142,1.420048)(0.144,1.284727)(0.146,1.133129)(0.148,0.959396)(0.15,0.749067)(0.152,0.454882)(0.154,0.054202)(0.156,0.012415)(0.158,0.008188)(0.16,0.006507)(0.162,0.005502)(0.164,0.004844)(0.166,0.004389)(0.168,0.004066)(0.17,0.003838)(0.172,0.003684)(0.174,0.003594)(0.176,0.003567)(0.178,0.003606)(0.18,0.003733)(0.182,0.003967)(0.184,0.004376)(0.186,0.005108)(0.188,0.006904)(0.19,0.073155)(0.192,0.312135)(0.194,0.485641)(0.196,0.518969)(0.198,0.589499)
(0.2,0.659299)(0.21,0.864995)(0.22,0.964776)(0.23,1.041109)(0.24,1.083337)(0.25,1.104571)(0.26,1.12575)(0.27,1.119406)(0.28,1.115421)(0.29,1.109176)(0.3,1.087731)(0.31,1.071227)(0.32,1.047058)(0.33,1.020041)(0.34,0.990463)(0.35,0.958563)(0.36,0.924495)(0.37,0.888339)(0.38,0.850104)(0.39,0.809726)(0.4,0.767057)(0.41,0.721852)(0.42,0.673731)(0.43,0.622129)(0.44,0.56618)(0.45,0.504517)(0.46,0.434808)(0.47,0.352583)(0.48,0.249154)(0.49,0.071485)(0.5,0.001538)(0.51,0.000808)(0.52,0.000631)(0.53,0.000541)(0.54,0.000491)(0.55,0.000468)(0.56,0.000465)(0.57,0.000487)(0.58,0.00055)(0.59,0.000708)(0.6,0.001715)(0.61,0.14419)(0.62,0.195779)(0.63,0.238064)(0.64,0.265859)(0.65,0.289627)(0.66,0.332697)(0.67,0.343112)(0.68,0.349127)(0.69,0.370253)(0.7,0.375966)(0.71,0.384661)(0.72,0.396067)(0.73,0.403184)(0.74,0.40885)(0.75,0.414243)(0.76,0.419252)(0.77,0.423685)(0.78,0.427407)(0.79,0.43052)(0.8,0.433433)(0.81,0.436074)(0.82,0.438108)(0.83,0.439965)(0.84,0.441422)(0.85,0.442664)(0.86,0.443591)(0.87,0.444308)(0.88,0.444804)(0.89,0.445085)(0.9,0.445176)(0.91,0.445094)(0.92,0.444849)(0.93,0.444452)(0.94,0.443913)(0.95,0.443242)(0.96,0.442448)(0.97,0.441538)(0.98,0.440521)(0.99,0.439403)(1.,0.438191)(1.01,0.436891)(1.02,0.435508)(1.03,0.434047)(1.04,0.432513)(1.05,0.430911)(1.06,0.429244)(1.07,0.427516)(1.08,0.425731)(1.09,0.423891)(1.1,0.422002)(1.11,0.420064)(1.12,0.41808)(1.13,0.416054)(1.14,0.413988)(1.15,0.411883)(1.16,0.409742)(1.17,0.407566)(1.18,0.405359)(1.19,0.40312)(1.2,0.400852)(1.21,0.398556)(1.22,0.396233)(1.23,0.393886)(1.24,0.391514)(1.25,0.38912)(1.26,0.386703)(1.27,0.384266)(1.28,0.381809)(1.29,0.379333)(1.3,0.376839)(1.31,0.374327)(1.32,0.371798)(1.33,0.369253)(1.34,0.366693)(1.35,0.364117)(1.36,0.361527)(1.37,0.358923)(1.38,0.356305)(1.39,0.353673)(1.4,0.351029)(1.41,0.348372)(1.42,0.345703)(1.43,0.343022)(1.44,0.340329)(1.45,0.337624)(1.46,0.334908)(1.47,0.33218)(1.48,0.329441)(1.49,0.326691)(1.5,0.32393)(1.51,0.321158)(1.52,0.318374)(1.53,0.31558)(1.54,0.312775)(1.55,0.309958)(1.56,0.30713)(1.57,0.304291)(1.58,0.30144)(1.59,0.298577)(1.6,0.295703)(1.61,0.292816)(1.62,0.289917)(1.63,0.287006)(1.64,0.284082)(1.65,0.281144)(1.66,0.278193)(1.67,0.275228)(1.68,0.272249)(1.69,0.269255)(1.7,0.266246)(1.71,0.263222)(1.72,0.260181)(1.73,0.257123)(1.74,0.254048)(1.75,0.250955)(1.76,0.247843)(1.77,0.244711)(1.78,0.24156)(1.79,0.238387)(1.8,0.235192)(1.81,0.231974)(1.82,0.228732)(1.83,0.225465)(1.84,0.222172)(1.85,0.218851)(1.86,0.215501)(1.87,0.212121)(1.88,0.208708)(1.89,0.205262)(1.9,0.201781)(1.91,0.198262)(1.92,0.194703)(1.93,0.191102)(1.94,0.187457)(1.95,0.183764)(1.96,0.180021)(1.97,0.176224)(1.98,0.172369)(1.99,0.168453)(2.,0.164471)(2.01,0.160417)(2.02,0.156287)(2.03,0.152073)(2.04,0.147768)(2.05,0.143364)(2.06,0.138852)(2.07,0.134219)(2.08,0.129453)(2.09,0.124539)(2.1,0.119458)(2.11,0.114187)(2.12,0.108699)(2.13,0.102959)(2.14,0.096921)(2.15,0.090527)(2.16,0.083696)(2.17,0.076312)(2.18,0.0682)(2.19,0.059079)(2.2,0.048449)(2.21,0.035377)(2.22,0.018759)(2.23,0.003855)
      };
      \legend{ {Empirical eigenvalue distribution},{Density of $\mu_N$} }
    \end{axis}
  \end{tikzpicture}
  \caption{Histogram of the eigenvalues of $\hat{C}_N$ for $n=2500$, $N=500$, $C_N=\diag(I_{125},3I_{125},10I_{250})$, $\tau_1$ with $\Gamma(.5,2)$-distribution.}
  \label{fig:hist_hatC}
\end{figure}

\begin{figure}[h!]
  \centering
  \begin{tikzpicture}[font=\footnotesize]
    \renewcommand{\axisdefaulttryminticks}{4} 
    \tikzstyle{every major grid}+=[style=densely dashed]       
    \tikzstyle{every axis y label}+=[yshift=-10pt] 
    \tikzstyle{every axis x label}+=[yshift=5pt]
    \tikzstyle{every axis legend}+=[cells={anchor=west},fill=white,
        at={(0.98,0.98)}, anchor=north east, font=\scriptsize ]
    \begin{axis}[
      xmin=0,
      ymin=0,
      xmax=2.3,
      ymax=3.5,
      bar width=1.5pt,
      grid=major,
      ymajorgrids=false,
      scaled ticks=true,
      xlabel={Eigenvalues},
      ylabel={Density}
      ]
      \addplot+[ybar,mark=none,color=black,fill=gray!40!white] coordinates{
	      (0.01,0.)(0.03,0.)(0.05,0.5)(0.07,3.1)(0.09,3.4)(0.11,2.8)(0.13,2.)(0.15,0.7)(0.17,0.)(0.19,0.3)(0.21,0.8)(0.23,1.)(0.25,1.1)(0.27,1.1)(0.29,1.1)(0.31,1.1)(0.33,1.2)(0.35,0.8)(0.37,0.9)(0.39,0.8)(0.41,0.8)(0.43,0.6)(0.45,0.4)(0.47,0.4)(0.49,0.1)(0.51,0.)(0.53,0.)(0.55,0.)(0.57,0.)(0.59,0.)(0.61,0.2)(0.63,0.2)(0.65,0.4)(0.67,0.2)(0.69,0.4)(0.71,0.4)(0.73,0.4)(0.75,0.5)(0.77,0.3)(0.79,0.5)(0.81,0.5)(0.83,0.3)(0.85,0.4)(0.87,0.5)(0.89,0.5)(0.91,0.4)(0.93,0.5)(0.95,0.5)(0.97,0.4)(0.99,0.4)(1.01,0.4)(1.03,0.4)(1.05,0.5)(1.07,0.5)(1.09,0.4)(1.11,0.3)(1.13,0.5)(1.15,0.3)(1.17,0.5)(1.19,0.4)(1.21,0.4)(1.23,0.4)(1.25,0.4)(1.27,0.3)(1.29,0.4)(1.31,0.4)(1.33,0.4)(1.35,0.3)(1.37,0.4)(1.39,0.4)(1.41,0.4)(1.43,0.3)(1.45,0.4)(1.47,0.3)(1.49,0.3)(1.51,0.3)(1.53,0.3)(1.55,0.3)(1.57,0.4)(1.59,0.3)(1.61,0.3)(1.63,0.3)(1.65,0.2)(1.67,0.3)(1.69,0.2)(1.71,0.3)(1.73,0.2)(1.75,0.2)(1.77,0.4)(1.79,0.2)(1.81,0.2)(1.83,0.2)(1.85,0.3)(1.87,0.2)(1.89,0.2)(1.91,0.2)(1.93,0.2)(1.95,0.2)(1.97,0.1)(1.99,0.2)(2.01,0.3)(2.03,0.1)(2.05,0.1)(2.07,0.1)(2.09,0.1)(2.11,0.)(2.13,0.2)(2.15,0.1)(2.17,0.1)(2.19,0.)(2.21,0.)(2.23,0.1)(2.25,0.)
      };
      \addplot[black,line width=1pt] plot coordinates{
(0.038,0.004341)(0.04,0.004936)(0.042,0.005697)(0.044,0.006703)(0.046,0.008128)(0.048,0.010218)(0.05,0.013754)(0.052,0.021054)(0.054,0.050016)(0.056,1.334095)(0.058,1.966732)(0.06,2.377708)(0.062,2.648892)(0.064,2.874291)(0.066,3.035344)(0.068,3.153864)(0.07,3.242067)(0.072,3.315482)(0.074,3.373319)(0.076,3.404909)(0.078,3.422671)(0.08,3.430764)(0.082,3.424367)(0.084,3.416687)(0.086,3.399365)(0.088,3.37194)(0.09,3.344324)(0.092,3.310287)(0.094,3.266717)(0.096,3.225626)(0.098,3.18039)(0.1,3.130975)(0.102,3.077707)(0.104,3.01793)(0.106,2.960965)(0.108,2.901403)(0.11,2.838894)(0.112,2.773256)(0.114,2.704535)(0.116,2.632925)(0.118,2.558654)(0.12,2.481919)(0.122,2.402873)(0.124,2.321645)(0.126,2.238363)(0.128,2.153477)(0.13,2.062702)(0.132,1.969039)(0.134,1.872718)(0.136,1.768689)(0.138,1.661782)(0.14,1.545184)(0.142,1.420048)(0.144,1.284727)(0.146,1.133129)(0.148,0.959396)(0.15,0.749067)(0.152,0.454882)(0.154,0.054202)(0.156,0.012415)(0.158,0.008188)(0.16,0.006507)(0.162,0.005502)(0.164,0.004844)(0.166,0.004389)(0.168,0.004066)(0.17,0.003838)(0.172,0.003684)(0.174,0.003594)(0.176,0.003567)(0.178,0.003606)(0.18,0.003733)(0.182,0.003967)(0.184,0.004376)(0.186,0.005108)(0.188,0.006904)(0.19,0.073155)(0.192,0.312135)(0.194,0.485641)(0.196,0.518969)(0.198,0.589499)
(0.2,0.659299)(0.21,0.864995)(0.22,0.964776)(0.23,1.041109)(0.24,1.083337)(0.25,1.104571)(0.26,1.12575)(0.27,1.119406)(0.28,1.115421)(0.29,1.109176)(0.3,1.087731)(0.31,1.071227)(0.32,1.047058)(0.33,1.020041)(0.34,0.990463)(0.35,0.958563)(0.36,0.924495)(0.37,0.888339)(0.38,0.850104)(0.39,0.809726)(0.4,0.767057)(0.41,0.721852)(0.42,0.673731)(0.43,0.622129)(0.44,0.56618)(0.45,0.504517)(0.46,0.434808)(0.47,0.352583)(0.48,0.249154)(0.49,0.071485)(0.5,0.001538)(0.51,0.000808)(0.52,0.000631)(0.53,0.000541)(0.54,0.000491)(0.55,0.000468)(0.56,0.000465)(0.57,0.000487)(0.58,0.00055)(0.59,0.000708)(0.6,0.001715)(0.61,0.14419)(0.62,0.195779)(0.63,0.238064)(0.64,0.265859)(0.65,0.289627)(0.66,0.332697)(0.67,0.343112)(0.68,0.349127)(0.69,0.370253)(0.7,0.375966)(0.71,0.384661)(0.72,0.396067)(0.73,0.403184)(0.74,0.40885)(0.75,0.414243)(0.76,0.419252)(0.77,0.423685)(0.78,0.427407)(0.79,0.43052)(0.8,0.433433)(0.81,0.436074)(0.82,0.438108)(0.83,0.439965)(0.84,0.441422)(0.85,0.442664)(0.86,0.443591)(0.87,0.444308)(0.88,0.444804)(0.89,0.445085)(0.9,0.445176)(0.91,0.445094)(0.92,0.444849)(0.93,0.444452)(0.94,0.443913)(0.95,0.443242)(0.96,0.442448)(0.97,0.441538)(0.98,0.440521)(0.99,0.439403)(1.,0.438191)(1.01,0.436891)(1.02,0.435508)(1.03,0.434047)(1.04,0.432513)(1.05,0.430911)(1.06,0.429244)(1.07,0.427516)(1.08,0.425731)(1.09,0.423891)(1.1,0.422002)(1.11,0.420064)(1.12,0.41808)(1.13,0.416054)(1.14,0.413988)(1.15,0.411883)(1.16,0.409742)(1.17,0.407566)(1.18,0.405359)(1.19,0.40312)(1.2,0.400852)(1.21,0.398556)(1.22,0.396233)(1.23,0.393886)(1.24,0.391514)(1.25,0.38912)(1.26,0.386703)(1.27,0.384266)(1.28,0.381809)(1.29,0.379333)(1.3,0.376839)(1.31,0.374327)(1.32,0.371798)(1.33,0.369253)(1.34,0.366693)(1.35,0.364117)(1.36,0.361527)(1.37,0.358923)(1.38,0.356305)(1.39,0.353673)(1.4,0.351029)(1.41,0.348372)(1.42,0.345703)(1.43,0.343022)(1.44,0.340329)(1.45,0.337624)(1.46,0.334908)(1.47,0.33218)(1.48,0.329441)(1.49,0.326691)(1.5,0.32393)(1.51,0.321158)(1.52,0.318374)(1.53,0.31558)(1.54,0.312775)(1.55,0.309958)(1.56,0.30713)(1.57,0.304291)(1.58,0.30144)(1.59,0.298577)(1.6,0.295703)(1.61,0.292816)(1.62,0.289917)(1.63,0.287006)(1.64,0.284082)(1.65,0.281144)(1.66,0.278193)(1.67,0.275228)(1.68,0.272249)(1.69,0.269255)(1.7,0.266246)(1.71,0.263222)(1.72,0.260181)(1.73,0.257123)(1.74,0.254048)(1.75,0.250955)(1.76,0.247843)(1.77,0.244711)(1.78,0.24156)(1.79,0.238387)(1.8,0.235192)(1.81,0.231974)(1.82,0.228732)(1.83,0.225465)(1.84,0.222172)(1.85,0.218851)(1.86,0.215501)(1.87,0.212121)(1.88,0.208708)(1.89,0.205262)(1.9,0.201781)(1.91,0.198262)(1.92,0.194703)(1.93,0.191102)(1.94,0.187457)(1.95,0.183764)(1.96,0.180021)(1.97,0.176224)(1.98,0.172369)(1.99,0.168453)(2.,0.164471)(2.01,0.160417)(2.02,0.156287)(2.03,0.152073)(2.04,0.147768)(2.05,0.143364)(2.06,0.138852)(2.07,0.134219)(2.08,0.129453)(2.09,0.124539)(2.1,0.119458)(2.11,0.114187)(2.12,0.108699)(2.13,0.102959)(2.14,0.096921)(2.15,0.090527)(2.16,0.083696)(2.17,0.076312)(2.18,0.0682)(2.19,0.059079)(2.2,0.048449)(2.21,0.035377)(2.22,0.018759)(2.23,0.003855)
      };
      \legend{ {Empirical eigenvalue distribution},{Density of $\mu_N$} }
    \end{axis}
  \end{tikzpicture}
  \caption{Histogram of the eigenvalues of $\hat{S}_N$ for $n=2500$, $N=500$, $C_N=\diag(I_{125},3I_{125},10I_{250})$, $\tau_1$ with $\Gamma(.5,2)$-distribution.}
  \label{fig:hist_hatS}
\end{figure}

\begin{figure}[h!]
  \centering
  \begin{tikzpicture}[font=\footnotesize]
    \renewcommand{\axisdefaulttryminticks}{4} 
    \tikzstyle{every major grid}+=[style=densely dashed]       
    \tikzstyle{every axis y label}+=[yshift=-10pt] 
    \tikzstyle{every axis x label}+=[yshift=5pt]
    \tikzstyle{every axis legend}+=[cells={anchor=west},fill=white,
        at={(0.98,0.98)}, anchor=north east, font=\scriptsize ]
    \begin{axis}[
      xmin=0,
      ymin=0,
      xmax=30,
      ymax=0.5,
      bar width=1pt,
      grid=major,
      ymajorgrids=false,
      scaled ticks=true,
      xlabel={Eigenvalues},
      ylabel={Density}
      ]
      \addplot+[ybar,mark=none,color=black,fill=gray!40!white] coordinates{
	      (0.1,0.)(0.3,0.27)(0.5,0.41)(0.7,0.31)(0.9,0.21)(1.1,0.11)(1.3,0.13)(1.5,0.13)(1.7,0.14)(1.9,0.12)(2.1,0.12)(2.3,0.1)(2.5,0.09)(2.7,0.09)(2.9,0.08)(3.1,0.07)(3.3,0.06)(3.5,0.04)(3.7,0.03)(3.9,0.04)(4.1,0.05)(4.3,0.05)(4.5,0.05)(4.7,0.05)(4.9,0.04)(5.1,0.07)(5.3,0.04)(5.5,0.04)(5.7,0.06)(5.9,0.05)(6.1,0.04)(6.3,0.06)(6.5,0.03)(6.7,0.06)(6.9,0.03)(7.1,0.05)(7.3,0.04)(7.5,0.04)(7.7,0.05)(7.9,0.04)(8.1,0.04)(8.3,0.02)(8.5,0.05)(8.7,0.04)(8.9,0.04)(9.1,0.02)(9.3,0.04)(9.5,0.04)(9.7,0.03)(9.9,0.02)(10.1,0.04)(10.3,0.03)(10.5,0.04)(10.7,0.03)(10.9,0.02)(11.1,0.03)(11.3,0.03)(11.5,0.02)(11.7,0.03)(11.9,0.03)(12.1,0.02)(12.3,0.02)(12.5,0.03)(12.7,0.02)(12.9,0.02)(13.1,0.03)(13.3,0.03)(13.5,0.01)(13.7,0.03)(13.9,0.02)(14.1,0.01)(14.3,0.03)(14.5,0.02)(14.7,0.03)(14.9,0.01)(15.1,0.01)(15.3,0.03)(15.5,0.01)(15.7,0.03)(15.9,0.01)(16.1,0.01)(16.3,0.02)(16.5,0.01)(16.7,0.02)(16.9,0.01)(17.1,0.01)(17.3,0.02)(17.5,0.01)(17.7,0.01)(17.9,0.01)(18.1,0.02)(18.3,0.02)(18.5,0.)(18.7,0.02)(18.9,0.01)(19.1,0.01)(19.3,0.)(19.5,0.01)(19.7,0.02)(19.9,0.)(20.1,0.01)(20.3,0.01)(20.5,0.02)(20.7,0.)(20.9,0.01)(21.1,0.01)(21.3,0.01)(21.5,0.01)(21.7,0.)(21.9,0.01)(22.1,0.01)(22.3,0.)(22.5,0.01)(22.7,0.01)(22.9,0.)(23.1,0.01)(23.3,0.01)(23.5,0.)(23.7,0.01)(23.9,0.)(24.1,0.)(24.3,0.)(24.5,0.01)(24.7,0.01)(24.9,0.)(25.1,0.01)(25.3,0.)(25.5,0.)(25.7,0.)(25.9,0.01)(26.1,0.)(26.3,0.)(26.5,0.)(26.7,0.)(26.9,0.)(27.1,0.)(27.3,0.)(27.5,0.01)(27.7,0.)(27.9,0.)(28.1,0.)(28.3,0.)(28.5,0.)(28.7,0.)(28.9,0.)(29.1,0.)(29.3,0.)(29.5,0.01)(29.7,0.)(29.9,0.)(30.1,0.)(30.3,0.)(30.5,0.)(30.7,0.)
      };
      \addplot[black,line width=1pt] plot coordinates{
	      (0.,0.000035)(0.01,0.000037)(0.02,0.000039)(0.03,0.000041)(0.04,0.000043)(0.05,0.000045)(0.06,0.000048)(0.07,0.000051)(0.08,0.000054)(0.09,0.000057)(0.1,0.000061)(0.11,0.000066)(0.12,0.000071)(0.13,0.000076)(0.14,0.000083)(0.15,0.000091)(0.16,0.0001)(0.17,0.00011)(0.18,0.000124)(0.19,0.00014)(0.2,0.000161)(0.21,0.000188)(0.22,0.000226)(0.23,0.000283)(0.24,0.000378)(0.25,0.000571)(0.26,0.001329)(0.27,0.17333)(0.28,0.272534)(0.29,0.327193)(0.3,0.365364)(0.31,0.390961)(0.32,0.411762)(0.33,0.427116)(0.34,0.438931)(0.35,0.448945)(0.36,0.454649)(0.37,0.458518)(0.38,0.460397)(0.39,0.461095)(0.4,0.460717)(0.41,0.460061)(0.42,0.457941)(0.43,0.455247)(0.44,0.452011)(0.45,0.44832)(0.46,0.444246)(0.47,0.439849)(0.48,0.435183)(0.49,0.430291)(0.5,0.425211)(0.51,0.419974)(0.52,0.41461)(0.53,0.409141)(0.54,0.403587)(0.55,0.398224)(0.56,0.392547)(0.57,0.386828)(0.58,0.381078)(0.59,0.375308)(0.6,0.369526)(0.61,0.363714)(0.62,0.357929)(0.63,0.352153)(0.64,0.346388)(0.65,0.34064)(0.66,0.334909)(0.67,0.329199)(0.68,0.323512)(0.69,0.317848)(0.7,0.312209)(0.71,0.306594)(0.72,0.301005)(0.73,0.29544)(0.74,0.2899)(0.75,0.284383)(0.76,0.278888)(0.77,0.273415)(0.78,0.267961)(0.79,0.262524)(0.8,0.257102)(0.81,0.251693)(0.82,0.246294)(0.83,0.240902)(0.84,0.235513)(0.85,0.230123)(0.86,0.224728)(0.87,0.219324)(0.88,0.213904)(0.89,0.208463)(0.9,0.202993)(0.91,0.197487)(0.92,0.191936)(0.93,0.186329)(0.94,0.180655)(0.95,0.174899)(0.96,0.169045)(0.97,0.163073)(0.98,0.156959)(0.99,0.150673)
	      (1.,0.144178)(1.1,0.062709)(1.2,0.121647)(1.3,0.135823)(1.4,0.140217)(1.5,0.140237)(1.6,0.137948)(1.7,0.134355)(1.8,0.130012)(1.9,0.125244)(2.,0.120249)(2.1,0.115151)(2.2,0.110028)(2.3,0.104929)(2.4,0.099881)(2.5,0.094898)(2.6,0.089986)(2.7,0.085141)(2.8,0.080352)(2.9,0.075603)(3.,0.070872)(3.1,0.066125)(3.2,0.061316)(3.3,0.056379)(3.4,0.051206)(3.5,0.04562)(3.6,0.039315)(3.7,0.032798)(3.8,0.036394)(3.9,0.040717)(4.,0.043609)(4.1,0.045661)(4.2,0.047161)(4.3,0.048269)(4.4,0.049084)(4.5,0.049673)(4.6,0.050083)(4.7,0.050349)(4.8,0.050499)(4.9,0.050552)(5.,0.050524)(5.1,0.05043)(5.2,0.050279)(5.3,0.05008)(5.4,0.049841)(5.5,0.049568)(5.6,0.049267)(5.7,0.04894)(5.8,0.048593)(5.9,0.048228)(6.,0.047849)(6.1,0.047458)(6.2,0.047056)(6.3,0.046645)(6.4,0.046228)(6.5,0.045806)(6.6,0.045379)(6.7,0.044949)(6.8,0.044517)(6.9,0.044083)(7.,0.043649)(7.1,0.043214)(7.2,0.04278)(7.3,0.042346)(7.4,0.041914)(7.5,0.041484)(7.6,0.041056)(7.7,0.04063)(7.8,0.040206)(7.9,0.039786)(8.,0.039368)(8.1,0.038953)(8.2,0.038542)(8.3,0.038134)(8.4,0.037729)(8.5,0.037328)(8.6,0.03693)(8.7,0.036537)(8.8,0.036146)(8.9,0.03576)(9.,0.035377)(9.1,0.034999)(9.2,0.034624)(9.3,0.034252)(9.4,0.033885)(9.5,0.033521)(9.6,0.033161)(9.7,0.032805)(9.8,0.032452)(9.9,0.032104)(10.,0.031759)(10.1,0.031417)(10.2,0.031079)(10.3,0.030745)(10.4,0.030415)(10.5,0.030087)(10.6,0.029764)(10.7,0.029444)(10.8,0.029127)(10.9,0.028813)(11.,0.028503)(11.1,0.028197)(11.2,0.027893)(11.3,0.027593)(11.4,0.027295)(11.5,0.027001)(11.6,0.02671)(11.7,0.026422)(11.8,0.026137)(11.9,0.025855)(12.,0.025576)(12.1,0.0253)(12.2,0.025026)(12.3,0.024756)(12.4,0.024488)(12.5,0.024223)(12.6,0.02396)(12.7,0.0237)(12.8,0.023443)(12.9,0.023188)(13.,0.022936)(13.1,0.022686)(13.2,0.022439)(13.3,0.022194)(13.4,0.021952)(13.5,0.021711)(13.6,0.021473)(13.7,0.021238)(13.8,0.021004)(13.9,0.020773)(14.,0.020544)(14.1,0.020317)(14.2,0.020092)(14.3,0.01987)(14.4,0.019649)(14.5,0.01943)(14.6,0.019213)(14.7,0.018998)(14.8,0.018786)(14.9,0.018575)(15.,0.018365)(15.1,0.018158)(15.2,0.017953)(15.3,0.017749)(15.4,0.017547)(15.5,0.017347)(15.6,0.017148)(15.7,0.016951)(15.8,0.016756)(15.9,0.016562)(16.,0.01637)(16.1,0.01618)(16.2,0.015991)(16.3,0.015803)(16.4,0.015618)(16.5,0.015433)(16.6,0.01525)(16.7,0.015069)(16.8,0.014889)(16.9,0.014711)(17.,0.014533)(17.1,0.014358)(17.2,0.014183)(17.3,0.01401)(17.4,0.013838)(17.5,0.013668)(17.6,0.013499)(17.7,0.013331)(17.8,0.013164)(17.9,0.012999)(18.,0.012834)(18.1,0.012671)(18.2,0.01251)(18.3,0.012349)(18.4,0.012189)(18.5,0.012031)(18.6,0.011874)(18.7,0.011718)(18.8,0.011563)(18.9,0.011409)(19.,0.011256)(19.1,0.011105)(19.2,0.010954)(19.3,0.010805)(19.4,0.010656)(19.5,0.010508)(19.6,0.010362)(19.7,0.010216)(19.8,0.010072)(19.9,0.009928)(20.,0.009786)(20.1,0.009644)(20.2,0.009503)(20.3,0.009364)(20.4,0.009225)(20.5,0.009087)(20.6,0.00895)(20.7,0.008813)(20.8,0.008678)(20.9,0.008544)(21.,0.00841)(21.1,0.008277)(21.2,0.008145)(21.3,0.008014)(21.4,0.007883)(21.5,0.007754)(21.6,0.007625)(21.7,0.007497)(21.8,0.007369)(21.9,0.007242)(22.,0.007116)(22.1,0.006991)(22.2,0.006866)(22.3,0.006742)(22.4,0.006618)(22.5,0.006495)(22.6,0.006373)(22.7,0.006251)(22.8,0.006129)(22.9,0.006008)(23.,0.005888)(23.1,0.005768)(23.2,0.005649)(23.3,0.00553)(23.4,0.005411)(23.5,0.005292)(23.6,0.005174)(23.7,0.005057)(23.8,0.004939)(23.9,0.004822)(24.,0.004705)(24.1,0.004588)(24.2,0.004471)(24.3,0.004354)(24.4,0.004238)(24.5,0.004121)(24.6,0.004004)(24.7,0.003887)(24.8,0.003769)(24.9,0.003652)(25.,0.003534)(25.1,0.003415)(25.2,0.003296)(25.3,0.003175)(25.4,0.003054)(25.5,0.002932)(25.6,0.002807)(25.7,0.002681)(25.8,0.002553)(25.9,0.002421)(26.,0.002286)(26.1,0.002146)(26.2,0.001999)(26.3,0.001846)(26.4,0.001683)(26.5,0.001517)(26.6,0.001376)(26.7,0.001314)(26.8,0.001302)(26.9,0.001303)(27.,0.001303)(27.1,0.001301)(27.2,0.001295)(27.3,0.001287)(27.4,0.001276)(27.5,0.001262)(27.6,0.001246)(27.7,0.001229)(27.8,0.001211)(27.9,0.001191)(28.,0.001172)(28.1,0.001152)(28.2,0.001132)(28.3,0.001112)(28.4,0.001093)(28.5,0.001074)(28.6,0.001055)(28.7,0.001035)(28.8,0.001015)(28.9,0.000994)(29.,0.000972)(29.1,0.000948)(29.2,0.000923)(29.3,0.000897)(29.4,0.000869)(29.5,0.000839)(29.6,0.000807)(29.7,0.000772)(29.8,0.000735)(29.9,0.000695)(30.,0.000651)(30.1,0.000604)(30.2,0.00055)
      };
      \legend{ {Empirical eigenvalue distribution},{Density of the deterministic equivalent}}
    \end{axis}
  \end{tikzpicture}
  \caption{Histogram of the eigenvalues of $\frac1n\sum_{i=1}^n x_ix_i^*$ for $n=2500$, $N=500$, $C_N=\diag(I_{125},3I_{125},10I_{250})$, $\tau_1$ with $\Gamma(.5,2)$-distribution.}
  \label{fig:hist_sample_cov}
\end{figure}

Figures~\ref{fig:hist_hatC} and \ref{fig:hist_hatS} depict the empirical histogram of the eigenvalues of $\hat{C}_N$ and $\hat{S}_N$, for $N=500$ and $n=2500$ with $u(t)=(1+\alpha)/(t+\alpha)$, $\alpha=0.1$, $C_N=\diag(I_{125},3I_{125},10I_{250})$, and $\tau_1,\ldots,\tau_n$ i.i.d. with $\Gamma(.5,2)$ distribution. In thick line is also depicted the density of $\mu_N$ in Corollary~\ref{cor:spectrum} which shows an accurate match to the empirical spectrum as predicted by \eqref{eq:spectrum_conv}. As a comparison, Figure~\ref{fig:hist_sample_cov} shows the empirical histogram of the eigenvalues of the sample covariance matrix $\frac1n\sum_{i=1}^n x_ix_i^*$ under the same parametrization against the deterministic equivalent density for this model in thick line \citep{ZHA09}. This graph presents a seemingly unbounded eigenvalue spectrum support (in fact bounded for each $N$ but growing with $N$) which is expected since $\tau_1$ has unbounded support. Also note the gain of separability in the spectrum of $\hat{C}_N$ which exhibits clearly three compacts subsets of eigenvalues, reminiscent of the three masses in the eigenvalue distribution of $C_N$, while $\frac1n\sum_{i=1}^n x_ix_i^*$ exhibits a single compact set of eigenvalues. This has important consequences from detection and estimation purposes in signal processing application of robust estimation.

\bigskip

In the next section, we present the proofs of Theorem~\ref{th:uniqueness} and Theorem~\ref{th:main}.

\section{Proof of the main results}
\label{sec:proofs}
For the sake of definition, we take all variables to be complex here although the arguments are also valid for real random variables. 
\subsection{Proof of Theorem~\ref{th:uniqueness}}

As mentioned in Section~\ref{sec:preliminaries}, we can assume without generality restriction that $C_N=I_N$. Indeed, if $\hat{C}_N$ is the unique solution to \eqref{eq:def_hatCN} assuming $C_N=I_N$, then, for any other choice of $C_N\succ 0$, $C_N^\frac12\hat{C}_NC_N^\frac12$ is the unique solution to the corresponding model in \eqref{eq:def_hatCN}. Hence, we only need to prove the result for $C_N=I_N$.

Consider a growing sequence $\{x_1,\ldots,x_n\}_{n=1}^\infty$ according to Assumption~\ref{ass:cN}. Since $|\{\tau_i=0\}|=n\nu_n(\{0\})<n(1-c_+)$ for all large $n$ a.s. (Assumption~\ref{ass:xi}--\ref{item:0}), $n-|\{\tau_i=0\}|>c_+n>N+1$ which, along with $z_1,\ldots,z_n$ being normalized Gaussian vectors, ensures that $\{x_1,\ldots,x_{j-1},x_{j+1},\ldots,x_n\}$ spans $\CC^N$ for all $j$ for all large $n$ a.s.
	As long as $n$ is large enough, we can therefore almost surely define $h=(h_1,\ldots,h_n)$ with $h_j:\RR_+^n\to \RR_+$ given by
	\begin{align*}
		h_j(q_1,\ldots,q_n) = \frac1N z_j^* \left(\frac1n \sum_{i\neq j} \tau_i v \left( \tau_i q_i \right) z_iz_i^* \right)^{-1} z_j.
	\end{align*}
	As shown in Section~\ref{sec:rewriting}, in order to show that $\hat{C}_N$ is uniquely defined, it suffices to show that there exists a unique $q_1,\ldots,q_n$ such that for each $j$, $q_j=h_j(q_1,\ldots,q_n)$. For this, we show first that $h$ satisfies the following properties with probability one:
	\begin{itemize}
		\item[(a)] Nonnegativity: For each $q_1,\ldots,q_n\geq 0$ and each $i$, $h_i(q_1,\ldots,q_n)>0$
		\item[(b)] Monotonicity: For each $q_1\geq q_1',\ldots,q_n\geq q_n'$ and each $i$, $h_i(q_1,\ldots,q_n)\geq h_i(q_1',\ldots,q_n')$
		\item[(c)] Scalability: For each $\alpha>1$ and each $i$, $\alpha h_i(q_1,\ldots,q_n)>h_i(\alpha q_1,\ldots,\alpha q_n)$.
	\end{itemize}
	Item (a) is obvious since the matrix inverse is well defined for all $n$ large and $z_i\neq 0$ almost surely. Item (b) follows from the fact that, for two Hermitian matrices $A\succeq B\succ 0$, $B^{-1}\succeq A^{-1}\succ 0$ (\cite[Corollary~7.7.4]{HOR85}), and from $v$ being non-increasing, entailing $h_i$ to be a non-decreasing function of each $q_j$. As for Item (c), it follows also from the previous matrix inverse relation and from $\psi$ being increasing, entailing in particular that, for $\alpha>1$, $\psi(\alpha q_i)>\psi(q_i)$ if $q_i\neq 0$ so that $v(\alpha q_i)>v(q_i)/\alpha$ for $q_i\geq 0$.

	According to Yates \cite[Theorem~2]{YAT95}, $h$ is then a {\it standard interference function} and, if there exists $q_1,\ldots,q_n$ such that for each $i$, $q_i>h_i(q_1,\ldots,q_n)$ (feasibility condition), then there is a unique $\{q_1,\ldots,q_n\}$ satisfying $q_i=h_i(q_1,\ldots,q_n)$ for each $i$, which is given by $q_i=\lim_{t\to\infty} q_i^{(t)}$ with $q^{(0)}_i\geq 0$ arbitrary and, for $t\geq 0$, $q_i^{(t+1)}=h_i(q_1^{(t)},\ldots,q_n^{(t)})$ (which would then conclude the proof). To obtain the feasibility condition, note that the function $q\mapsto \frac1Nz_j^* \left( \frac1n \sum_{i\neq j} \psi(\tau_i q) z_iz_i^* \right)^{-1}z_j$ is decreasing and, as $q\to\infty$, has limit $\frac{1-c_N\phi_\infty}{\phi_\infty} \frac1Nz_j^* \left( \frac1n \sum_{i\neq j,\tau_i\neq 0} z_iz_i^* \right)^{-1}z_j$.
	As $\{\tau_i\}_{i=1}^n$ and $\{z_i\}_{i=1}^n$ are independent and $\limsup_n N/|\{\tau_i\neq 0\}|=\limsup c_N/(1-\nu_n(\{0\}))<1$ a.s. (Assumption~\ref{ass:xi} and Assumption~\ref{ass:cN}), for all large $n$ a.s., we fall within the hypotheses of Lemma~\ref{le:uniform_conv_sum_zi} in the Appendix and we can then write,\footnote{To be more exact, since $|\{\tau_i\neq 0\}|$ is random with probability space $\mathcal T$ producing the $\tau_i$'s, Lemma~\ref{le:uniform_conv_sum_zi} applies only on a subset of probability one of $\mathcal T$. It then suffices to apply Tonelli's theorem \citep{BIL08} to ensure that Lemma~\ref{le:uniform_conv_sum_zi} can be extended and still holds with probability one on the product space producing the $(\tau_i,z_i)$.}
	\begin{align*}
		\max_{1\leq j\leq n}\left|(1-\nu_n(\{0\}))\frac1Nz_j^* \left( \frac1n \sum_{\tau_i\neq 0} z_iz_i^* \right)^{-1}z_j-1\right|\asto 0.
	\end{align*}
	Assume first that $\tau_j\neq 0$. Then, using the relation
	\begin{align*}
		\frac1Nz_j^* \left( \frac1n \sum_{\tau_i\neq 0,i\neq j} z_iz_i^* \right)^{-1}z_j = \frac{\frac1Nz_j^* \left( \frac1n \sum_{\tau_i\neq 0} z_iz_i^* \right)^{-1}z_j}{1-c_N \frac1Nz_j^* \left( \frac1n \sum_{\tau_i\neq 0} z_iz_i^* \right)^{-1}z_j}
	\end{align*}
	and the fact that for all large $n$ a.s. $1-\nu_n(\{0\})>c_+$, we have
	\begin{align*}
		\max_{j,\tau_j\neq 0}\left|\frac1Nz_j^* \left( \frac1n \sum_{\tau_i\neq 0,i\neq j} z_iz_i^* \right)^{-1}z_j- \frac1{1-\nu_n(\{0\})-c_N}\right|\asto 0.
	\end{align*}
	Therefore, using the fact that $\nu_N(\{0\})<1-\phi_\infty^{-1}$ for all $n$ large a.s. (Assumption~\ref{ass:xi}--\ref{item:0}), we have that for all $j$ with $\tau_j\neq 0$
	\begin{align}
		\label{eq:ineq_uniq}
		\frac{1-c_N\phi_\infty}{\phi_\infty} \frac1Nz_j^* \left( \frac1n \sum_{\tau_i\neq 0,i\neq j} z_iz_i^* \right)^{-1}z_j < 1.
	\end{align}
	If instead $\tau_j=0$, then
	\begin{align*}
		\max_{j,\tau_j=0}\left|\frac1Nz_j^* \left( \frac1n \sum_{\tau_i\neq 0} z_iz_i^* \right)^{-1}z_j- \frac1{1-\nu_n(\{0\})}\right|\asto 0.
	\end{align*}
	and we find also the inequality \eqref{eq:ineq_uniq} for all large $n$ a.s. and for all $j$ with $\tau_j=0$, using once more $\nu_N(\{0\})<1-\phi_\infty^{-1}$. As such, \eqref{eq:ineq_uniq} is valid for all $j\in\{1,\ldots,n\}$.

	We can then choose $n$ large enough so that \eqref{eq:ineq_uniq} holds for all $j$, after what, taking $q$ sufficiently large,
	\begin{align*}
		\frac1Nz_j^* \left( \frac1n \sum_{i\neq j} \psi(\tau_i q) z_iz_i^* \right)^{-1}z_j<1
	\end{align*}
	which is equivalent to
	\begin{align*} 
		\frac1Nz_j^* \left( \frac1n \sum_{i\neq j} v(\tau_i q) \tau_i z_iz_i^* \right)^{-1}z_j < q
	\end{align*}
	for all $j$, i.e. $h_j(q,\ldots,q)<q$. This ensures feasibility for all large $n$ a.s. and concludes the proof.

\subsection{Proof of Theorem~\ref{th:main}}

Similar to the proof of Theorem~\ref{th:uniqueness}, we can restrict ourselves to the assumption that $C_N=I_N$. The generalization to $C_N$ satisfying Assumption~\ref{ass:xi}-\ref{item:CN}) will follow straightforwardly. We therefore take $C_N=I_N$ in what follows.

We start the proof by introducing the following fundamental lemmas (note that these lemmas in fact hold true irrespective of $C_N\succ 0$).

\begin{lemma}
	\label{le:gamma_N}
	Let Assumption~\ref{ass:cN} hold and let $h:[0,\infty)\to [0,\infty)$ be given by
	\begin{align*}
		h(\gamma) = \left( \frac1n \sum_{i=1}^n \frac{\tau_i v(\tau_i \gamma) }{1+c_N\tau_i v(\tau_i \gamma) \gamma} \right)^{-1} = \left\{ 
			\begin{array}{ll}
				\gamma \left( \frac1n\sum_{i=1}^n \frac{\psi(\tau_i \gamma)}{1+c_N\psi(\tau_i\gamma)} \right)^{-1} &,~\gamma>0 \\
				\frac1{v(0)}\left( \frac1n\sum_{i=1}^n \tau_i \right)^{-1}&,~\gamma=0.
			\end{array}
			\right.
	\end{align*}
	Then, for all large $n$ a.s., there exists a unique $\gamma_N>0$ satisfying $\gamma_N=h(\gamma_N)$, given by
	\begin{align*}
		\gamma_N = \lim_{t\to\infty} \gamma_N^{(t)}
	\end{align*}
	with $\gamma_N^{(0)}\geq 0$ arbitrary and, for $t\geq 0$, $\gamma_N^{(t+1)} = h(\gamma_N^{(t)})$. Moreover, with probability one, 
	\begin{align*}
\gamma_-<\liminf_N \gamma_N\leq \limsup_N \gamma_N <\gamma_+
	\end{align*}
	for some $\gamma_-,\gamma_+>0$ finite.
\end{lemma}
\begin{proof}
	As in the proof of Theorem~\ref{th:uniqueness}, we show that $h$ (scalar-valued this time) is a standard interference function. We show easily positivity, monotonicity and scalability of $h$. Indeed, for $\gamma\geq 0$, $h(\gamma)>0$. For $\gamma\geq \gamma'\geq 0$, 
	\begin{align*}
		\frac{h(\gamma)-h(\gamma')}{h(\gamma)h(\gamma')} &= \frac1n\sum_{i=1}^n \frac{\tau_i \left( v(\tau_i \gamma')- v(\tau_i \gamma) \right) + (\gamma-\gamma')c_N \tau_i^2v(\tau_i \gamma)v(\tau_i \gamma')}{(1+c_N\tau_i v(\tau_i \gamma)\gamma)(1+c_N\tau_i v(\tau_i \gamma')\gamma')} \geq 0
	\end{align*}
	which follows from $v$ being nonnegative decreasing. Finally, for $\alpha>1$, $\alpha h(0)>h(0)$ and for $\gamma\neq 0$, 
	\begin{align*}
		h(\alpha \gamma)&=\alpha \gamma \left( \frac1n\sum_{i=1}^n \frac{\psi(\tau_i \alpha \gamma)}{1+c_N \psi(\tau_i \alpha \gamma)} \right)^{-1} < \alpha \gamma \left( \frac1n\sum_{i=1}^n \frac{\psi(\tau_i \gamma)}{1+c_N \psi(\tau_i \gamma)} \right)^{-1} \\
		&= \alpha h(\gamma)
	\end{align*}
	which follows from $\gamma\mapsto \psi(\tau_i \gamma)(1+c_N \psi(\tau_i \gamma))^{-1}$ being increasing as long as $\tau_i\neq 0$. It remains to prove the existence of a $\gamma$ such that $\gamma>h(\gamma)$, inducing by \cite[Theorem~2]{YAT95} the uniqueness of the fixed-point $\gamma_N$ given by $\gamma_N = \lim_{t\to\infty} \gamma_N^{(t)}$ as stated in the theorem. For this, we use again the fact that $\gamma\mapsto \psi(\tau_i \gamma)(1+c_N \psi(\tau_i \gamma))^{-1}$ is increasing and that (Assumption~\ref{ass:xi}--\ref{item:0}), for all large $n$ a.s.
\begin{align*}
	\lim_{\gamma\to\infty} \frac1n\sum_{i=1}^n \frac{\psi(\tau_i \gamma)}{1+c_N\psi(\tau_i \gamma)} = \frac{(1-\nu_n(\{0\})) \psi^N_\infty }{1+c_N \psi^N_\infty} = (1-\nu_n(\{0\})) \phi_\infty > 1.
\end{align*}
Therefore, there exists $\gamma_0$ (a priori dependent on the set $\{\tau_1,\ldots,\tau_n\}$) such that, for all $\gamma>\gamma_0$, $h(\gamma)< \gamma$. 

To prove uniform boundedness of $\gamma_N$, let $\varepsilon>0$ and $m>0$ be such that $(1-\varepsilon)\phi_\infty > 1$ and $\nu_n( (m,\infty) )>1-\varepsilon$ for all $n$ large a.s. (always possible from Assumption~\ref{ass:xi}--\ref{item:0}). Then, for all $n$ large a.s.
\begin{align*}
	\frac1n\sum_{i=1}^n \frac{\psi(\tau_i \gamma)}{1+c_N\psi(\tau_i \gamma)} > (1-\varepsilon) \frac{\psi(m \gamma)}{1+c_N\psi(m\gamma)} \to (1-\varepsilon)\phi_\infty > 1
\end{align*}
as $\gamma\to\infty$.
Similar to $\gamma_0$ above, we can therefore choose $\gamma_+$ large enough, now independent of $n$ large, such that, a.s. $\gamma\geq \gamma_+ \Rightarrow \gamma>h(\gamma)$, implying $\gamma_N<\gamma_+$ for these $n$ large since $\gamma_N=h(\gamma_N)$. Also, $h(0)>1/(2v(0))$ for all large $n$ a.s. since $\frac1n\sum_{i=1}^n \tau_i\asto 1$ by Assumption~\ref{ass:xi}. Hence, by the continuous growth of $h$, we can take $\gamma_-=1/(2v(0))>0$ which is such that $\gamma\leq \gamma_- \Rightarrow h(\gamma)\geq h(0)>\gamma$ for all large $n$ a.s. This implies $\gamma_N>\gamma_-$ for all large $n$ a.s., which concludes the proof.
\end{proof}

\begin{remark}
	\label{rem:lemma1_eta}
	For further use, note that Lemma~\ref{le:gamma_N} can be refined as follows. Let $(\eta,M_\eta)$ be couples indexed by $\eta$ with $0<\eta<1$ and $M_\eta>0$ such that $\nu_n( (M_\eta,\infty) )<\eta$ for all large $n$ a.s. (possible by tightness of $\nu_n$). Then, for sufficiently small $\eta$, the equation in $\gamma$
	\begin{align*}
		\gamma = \left( \frac1n \sum_{\tau_i\leq M_\eta} \frac{\tau_i v(\tau_i \gamma)}{1+c_N \tau_i v(\tau_i \gamma)\gamma} \right)^{-1}
	\end{align*}
	has a unique solution $\gamma_N^\eta$ for all large $n$ a.s. and there exists $\gamma_-,\gamma_+>0$ such that, for all $\eta<\eta_0$ small, $\gamma_-<\gamma_N^\eta<\gamma_+$ for all large $n$ a.s. 
\end{remark}
\begin{proof}
	The uniqueness is clear as long as $(1-\eta_0)(1-\limsup_n\nu_n(\{0\}))\phi_\infty>1$ since then, exploiting the fact that $\lim_n \frac{|\{\tau_i\leq M_\eta\}|}{n}>1-\eta_0$ a.s., 
\begin{align*}
	\lim_{\gamma\to\infty} \frac1n\sum_{\tau_i\leq M_\eta} \frac{\psi(\tau_i \gamma)}{1+c_N\psi(\tau_i \gamma)} = \frac{|\{\tau_i\leq M_\eta\}|}{n} (1-\nu_n(\{0\})) \phi_\infty > 1
\end{align*}
for all $n$ large a.s. and the proof follows from the proof of Lemma~\ref{le:gamma_N}. 
For uniform boundedness, taking $M_{\eta_0}<M_\eta$ large enough (or equivalently $\eta_0>\eta$ small enough) such that $\liminf_n \frac{|\{m<\tau_i\leq M_\eta\}|}n>\liminf_n \frac{|\{m<\tau_i\leq M_{\eta_0}\}|}n>1-\varepsilon$ a.s. in the proof of Lemma~\ref{le:gamma_N} leads to the same upper bound result for all small $\eta<\eta_0$. As for the lower bound, we still have $h(0)>1/(2v(0))$ for all large $n$ a.s. independently of $\eta$ so the result is maintained.
\end{proof}

\begin{lemma}
	\label{le:convergence_gamma}
	Let Assumption~\ref{ass:cN} hold and define $\gamma_N$ as in Lemma~\ref{le:gamma_N}. Then, as $n\to\infty$,
	\begin{align*}
		\max_{1\leq j\leq n} \left| \frac1N z_j^* \left( \frac1n \sum_{i\neq j} \tau_i v(\tau_i \gamma_N) z_iz_i^* \right)^{-1} z_j - \gamma_N \right| \asto 0.
	\end{align*}
\end{lemma}
\begin{proof}
	We first introduce some notations to simplify readability. First, we will write $z_j=\sqrt{\bar N}A_N\tilde{y}_j/\Vert \tilde{y}_j\Vert\triangleq \sqrt{\bar N}\tilde{z}_j/\Vert \tilde{y}_j\Vert$ with $\tilde{y}_j$ zero-mean $I_{\bar N}$-covariance Gaussian, hence $\tilde{z}_j$ zero-mean $I_N$-covariance Gaussian. With this notation, in what follows, we denote $A=\frac1n\sum_{i=1}^n\tau_iv(\tau_i\gamma_N)z_iz_i^*$, $A_{(j)}=A-\frac1n\tau_jv(\tau_j\gamma_N)z_jz_j^*$, $\tilde{A}=\frac1n\sum_{i=1}^n\tau_iv(\tau_i\gamma_N)\tilde{z}_i\tilde{z}_i^*$ and $\tilde{A}_{(j)}=\tilde{A}-\frac1n\tau_jv(\tau_j\gamma_N)\tilde{z}_j\tilde{z}_j^*$.

We first show that there exists $\eta>0$ such that, for all large $n$ a.s.
\begin{align}
	\label{eq:l1_away}
	\min_{1\leq j\leq n} \lambda_1 \left( A_{(j)} \right) > \eta
\end{align}
(recall that $\lambda_1$ stands for the smallest eigenvalue). For this, take $0<\varepsilon<1-c_+$ and $m>0$ be such that $\nu_n( (m,\infty) )>1-\varepsilon$ for all $n$ large a.s. (Assumption~\ref{ass:xi}--\ref{item:0}).
Using the fact that $xv(x)=\psi(x)$ is non-decreasing and that any subtraction of a nonnegative definite matrix cannot increase the smallest eigenvalue, we have
\begin{align}
\label{eq:lmin}	
\min_{1\leq j\leq n} \lambda_1 \left( A_{(j)}\right) &\geq \min_{1\leq j\leq n} \lambda_1 \left( \frac1n \sum_{i\neq j,\tau_i\geq m} \frac{\psi(\tau_i \gamma_N)}{\gamma_N} z_iz_i^* \right) \nonumber \\
&\geq \frac{\psi(m\gamma_N)}{\gamma_N} \min_{1\leq j\leq n} \lambda_1 \left( \frac1n \sum_{i\neq j,\tau_i \geq m} z_iz_i^* \right).
\end{align}
Since $\nu_n( (m,\infty) )>1-\varepsilon$ for all $n$ large a.s.,
\begin{align*}
	0<c_-<\liminf_n \frac{N}{|\{\tau_i\geq m\}|}\leq \limsup_n \frac{N}{|\{\tau_i\geq m\}|}< \frac{c_+}{1-\varepsilon}<1.
\end{align*}
From Lemma~\ref{le:uniform_conv_sum_zi} in the Appendix (see footnote in the proof of Theorem~\ref{th:uniqueness} for details), we can then write
\begin{align*}
	\min_{1\leq j\leq n} \lambda_1 \left( A_{(j)}\right) &\geq \frac{\psi(m\gamma_N)}{\gamma_N} \nu_n( (m,\infty) ) \min_{1\leq j\leq n} \lambda_1 \left( \frac1{|\{\tau_i\geq m\}|} \sum_{i\neq j,\tau_i \geq m} z_iz_i^* \right) \\
	&> \frac{\psi(m\gamma_N)}{\gamma_N} (1-\varepsilon)\eta'
\end{align*}
for some $\eta'>0$ which, along with the almost sure boundedness of $\gamma_N$ (Lemma~\ref{le:gamma_N}) proves \eqref{eq:l1_away}. 

Now that \eqref{eq:l1_away} is acquired, let $\EE_{z_j}$ denote the expectation with respect to $z_j$ (i.e. conditionally on the sigma-field engendered by the $z_i$, $i\neq j$, and the $\tau_i$) and $\kappa_j\triangleq 1_{ \{\lambda_1(A_{(j)})> \eta\}}$ with $\eta$ as defined in \eqref{eq:l1_away}.
From \cite[Lemma~B.26]{SIL06} (which applies here since $\tilde{z}_j$ and $\kappa_j^{1/p} A_{(j)}^{-1}$ are independent), for $p>2$,
\begin{align*}
	\EE_{\tilde{y}_j} \left[ \kappa_j \left| \frac1N \tilde{z}_j^* A_{(j)}^{-1} \tilde{z}_j - \frac1N\tr A_{(j)}^{-1} \right|^p \right] \leq \frac{\kappa_j K_p}{N^{\frac{p}2} } \left[ \left( \frac{\zeta_4}N \tr A_{(j)}^{-2} \right)^{\frac{p}2} + \frac{\zeta_{2p}}{N^{\frac{p}2}} \tr A_{(j)}^{-p}\right]
\end{align*}
for $\zeta_\ell$ any upper bound on $\EE[|\tilde{z}_{ij}|^\ell]$ and $K_p$ a constant dependent only on $p$. From the definition of $\kappa_j$, we have $\kappa_j \Vert A_{(j)}^{-1}\Vert < \eta^{-1}$, so that, using $\frac1N\tr B\leq \Vert B\Vert$ for nonnegative definite $B\in\CC^{N\times N}$,
\begin{align*}
	\EE_{\tilde{y}_j} \left[ \kappa_j \left| \frac1N \tilde{z}_j^* A_{(j)}^{-1} \tilde{z}_j - \frac1N\tr A_{(j)}^{-1} \right|^p \right] \leq \frac{K_p}{\eta^p N^{\frac{p}2}}\left(\zeta_4^{\frac{p}2} + \frac{\zeta_{2p}}{N^{\frac{p}2-1}} \right).
\end{align*}
This bound being irrespective of all $z_i$ and $\tau_i$, $i\neq j$, we can take the expectation with respect to all $y_i$, $i\neq j$, and all $\tau_i$ to obtain
\begin{align*}
	\EE \left[ \kappa_j \left| \frac1N \tilde{z}_j^* A_{(j)}^{-1} \tilde{z}_j - \frac1N\tr A_{(j)}^{-1} \right|^p \right] = \mathcal O\left( \frac1{N^{\frac{p}2}} \right).
\end{align*}
Taking $p>4$ and applying the union bound, Markov inequality, and Borel Cantelli lemma finally shows that
\begin{align}
	\label{eq:conv_trace}
	\max_{1\leq j\leq n}  \kappa_j \left| \frac1N \tilde{z}_j^* A_{(j)}^{-1} \tilde{z}_j - \frac1N\tr A_{(j)}^{-1} \right| \asto 0.
\end{align}

With the same arguments on $\kappa_j$ and with the same $p$ as above, now remark that
\begin{align*}
	\EE_{\tilde{y}_j} \left[ \kappa_j \left| \frac1N z_j^* A_{(j)}^{-1} z_j - \frac1N \tilde{z}_j^* A_{(j)}^{-1} \tilde{z}_j \right|^p \right] &= \EE_{\tilde{y}_j} \left[ \kappa_j \left| \frac1N z_j^* A_{(j)}^{-1} z_j \left( 1 - \frac{\Vert \tilde{y}_j\Vert^2}{\bar N} \right) \right|^p \right] \\
	&\leq \frac1{\eta^p} \EE_{\tilde{y}_j} \left[ \left| 1 - \frac{\Vert \tilde{y}_j\Vert^2}{\bar N} \right|^p \right] = \mathcal O\left( \frac1{N^{p/2}} \right)
\end{align*}
since $\bar{N}\geq N$. Therefore, by the union bound, Markov inequality, and Borel Cantelli lemma,
\begin{align}
	\label{eq:conv_Gauss}
	\max_{1\leq j\leq n}  \kappa_j\left| \frac1N z_j^* A_{(j)}^{-1} z_j - \frac1N \tilde{z}_j^* A_{(j)}^{-1} \tilde{z}_j \right| \asto 0.
\end{align}

Combining \eqref{eq:conv_trace} and \eqref{eq:conv_Gauss} along with the fact that $\min_{1\leq j\leq n} \kappa_j \asto 1$ (from \eqref{eq:l1_away}) finally gives
\begin{align*}
	\max_{1\leq j\leq n} \left| \frac1N z_j^* A_{(j)}^{-1} z_j - \frac1N\tr A_{(j)}^{-1} \right| \asto 0.
\end{align*}

By \eqref{eq:l1_away}, $A_{(j)}=(A_{(j)}-\frac{\eta}2I_N)+\frac{\eta}2I_N$ with $\liminf_n \lambda_1(A_{(j)}-\frac{\eta}2I_N)>0$ a.s., so we are in the conditions of Lemma \ref{le:rank1perturbation} and we have
\begin{align*}
	\max_{1\leq j\leq n} \left|\frac1N\tr A_{(j)}^{-1} - \frac1N\tr A^{-1} \right| \asto 0.
\end{align*}

It remains to find a deterministic equivalent for $\frac1N\tr A^{-1}$. Similar to above, note first that, for all large $n$ a.s.
\begin{align*}
	\left| \frac1N\tr A^{-1} - \frac1N\tr \tilde{A}^{-1} \right| \leq \frac1{\eta^2} \frac{\psi_\infty}{\gamma_N} \max_{1\leq j\leq n}\left|\frac{1-{\bar N}^{-1}\Vert \tilde{y}_j\Vert^2}{ {\bar N}^{-1}\Vert \tilde{y}_j \Vert^2} \right| \left\Vert \frac1n\sum_{i=1}^n \tilde{z}_j\tilde{z}_j^*\right\Vert
\end{align*}
where we used the definition and boundedness of $\psi$ and standard matrix inversion formulas. From \citep{SIL98}, the right hand side converges almost surely to zero, so that it is equivalent to consider $z_i$ or $\tilde{z}_i$. Now, the trace $\frac1N\tr \tilde{A}^{-1}$ is exactly the Stieltjes transform $\hat m_N(z)$ of the matrix $\tilde{A}$ evaluated at point $z=0$. Since $\lambda_1( \tilde{A} )\geq \lambda_1( \tilde{A}_{(1)} )>\eta$ for all large $n$ a.s. and since $\tau_i v(\tau_i \gamma_N)=\psi(\tau_i\gamma_N)\gamma_N^{-1}$ is uniformly bounded across $i$ and $n$ (from the boundedness of $\psi$ and Lemma~\ref{le:gamma_N}), from standard random matrix results (e.g. \citep{COU09})\footnote{
	More precisely, \citep{COU09} shows that $\hat{m}_N(z)-m_N(z)\asto 0$ for all points $z$ with $\Im[z]>0$. Using $\lambda_1(\tilde{A})>\eta$ for all large $n$ a.s., the proof can be generalized to all $z\in \CC$ with positive distance to $[\eta,\infty)$ by turning the bounds in $1/|\Im[z]|$ into $1/d(z,[\eta,\infty))$ with $d$ denoting the Hausdorff distance, so in particular to $z=0$. The convergence $\hat{m}_N(0)-m_N(0)\asto 0$ is in particular already obtained in the generalization of the existence result of \citep[Appendix~A-C]{COU09}, this time for $z=0$.
	The proof of uniqueness of $m_N(0)$ can then again be checked by standard interference function arguments, where feasibility follows in particular from the right-hand behaving as $c_Nm < m$ from Assumption~\ref{ass:cN}. 
}, we have 
\begin{align*}
	\hat m_N(0) - m_N(0) \asto 0
\end{align*}
where $m_N(0)$ is the unique nonnegative solution to the equation in $m$ (as long as at least one $\tau_i$ is non-zero)
\begin{align*}
	m = \left( \frac1n \sum_{i=1}^n \frac{\tau_i v(\tau_i \gamma_N)}{1+c_N \tau_i v(\tau_i \gamma_N) m} \right)^{-1}.
\end{align*}
Now, by definition, $\gamma_N$ coincides with such a solution. By uniqueness of $m_N(0)$, one must then have $m_N(0)=\gamma_N$ so that, gathering all results together,
\begin{align*}
	\max_{1\leq j\leq n} \left|\frac1N z_j^* A_{(j)}^{-1} z_j - \gamma_N \right| \asto 0
\end{align*}
which completes the proof.
\end{proof}

\begin{remark}
	\label{rem:lemma2_eta}
	Similar to Remark~\ref{rem:lemma1_eta}, note that Lemma~\ref{le:convergence_gamma} can be further extended to 
	\begin{align*}
		\max_{1\leq j\leq n} \left| \frac1N z_j^* \left( \frac1n \sum_{\tau_i \leq M_\eta,i\neq j} \tau_i v(\tau_i \gamma_N^\eta) z_iz_i^* \right)^{-1}z_j - \gamma_N^\eta \right| \asto 0
	\end{align*}
	for some $\eta$ small enough, with $M_\eta$ and $\gamma_N^\eta$ defined in Remark~\ref{rem:lemma1_eta}.
\end{remark}
\begin{proof}
	One shows boundedness of $\lambda_1(\frac1n\sum_{\tau_i\leq M_\eta,i\neq j} \tau_i v(\tau_i \gamma^\eta_N)z_iz_i^*)$ simply by taking $\eta$ for which $\nu_n( (m,M_\eta) )>1-\varepsilon$ for all large $n$ a.s. in the proof of Lemma~\ref{le:convergence_gamma}. Then it suffices to adapt all derivations by substituting $\tau_i$ by zero if $\tau_i>M_\eta$. The result follows straightforwardly. 
\end{proof}

The two lemmas above are standard random matrix results on $x_1,\ldots,x_n$, independent of the structure of $\hat{C}_N$. The next lemma introduces a first result on the matrix $\hat{C}_N$ which will be fundamental in what follows. Recall that we denoted $d_i=\frac1Nz_i^*\hat{C}_{(i)}^{-1}z_i$, with $\hat{C}_{(i)}=\hat{C}_N-\frac1n v(\tau_id_i)\tau_i z_iz_i^*$.
\begin{lemma}[Boundedness of the $d_i$]
	\label{le:di_bounded}
There exist $d_+>d_->0$ such that, for all large $n$ a.s.,
\begin{align*}
	d_-<\liminf_n \min_{1\leq i\leq n} d_i \leq \limsup_n \max_{1\leq i\leq n} d_i < d_+.
\end{align*}
\end{lemma}
\begin{proof}
	Let us denote $d_{\max}=\max_{1\leq i \leq n}d_i$ and $d_{\min}=\min_{1\leq i \leq n}d_i$.
	Take $j\in\{1,\ldots,n\}$ arbitrary and, for $0<\varepsilon<1-\phi_\infty^{-1}<1-c_+$, take $m>0$ such that for all large $n$ a.s. $\nu_n( [m,\infty) )>1-\varepsilon$ (Assumption~\ref{ass:xi}--\ref{item:0}). Then, using the fact that $v$ is non-increasing while $\psi$ is non-decreasing,
	\begin{align}
		\hat{C}_{(j)} &\succeq \frac1n \sum_{\substack{i\neq j \\ \tau_i\geq m}} \tau_i v(\tau_i d_i) z_iz_i^* = \frac1n \sum_{\substack{i\neq j \\ \tau_i\geq m}} \frac{\psi(\tau_id_i)}{d_i} z_iz_i^* \succeq  \frac1n \sum_{\substack{i\neq j \\ \tau_i\geq m}} \frac{\psi(m d_i)}{d_i} z_iz_i^* \nonumber \\
		&= \frac1n \sum_{\substack{i\neq j \\ \tau_i\geq m}} m v(m d_i) z_iz_i^* \succeq m v(m d_{\max}) \frac1n \sum_{\substack{i\neq j \\ \tau_i\geq m}}  z_iz_i^*.
		\label{eq:ineq_md}
	\end{align}
	The right-hand side matrix is invertible for $n$ large since $|\{\tau_i\geq m\}|>nc_+>N$ for all large $n$ a.s.
	Therefore, choosing $j$ to be such that $d_{\max}=\frac1Nz_j^*\hat{C}_{(j)}^{-1}z_j$, and using $A\succeq B\succ 0 \Rightarrow B^{-1}\succeq A^{-1}$ for Hermitian $A,B$ matrices,
	\begin{align*}
		d_{\max} \leq \frac1{m v(m d_{\max})} \frac1Nz_j^*\left( \frac1n \sum_{\tau_i\geq m,i\neq j}  z_iz_i^* \right)^{-1}z_j.
	\end{align*}
	This implies
	\begin{align*}
		\psi(md_{\max}) \leq \frac1Nz_j^*\left( \frac1n \sum_{\tau_i\geq m,i\neq j}  z_iz_i^* \right)^{-1}z_j
	\end{align*}
	which can be rewritten, from the definition of $\psi$,
	\begin{align}
		\label{eq:psi_ineq}
		\phi( g^{-1}( md_{\max} )) \leq \frac{\frac1Nz_j^*\left( \frac1n \sum_{\tau_i\geq m,i\neq j}  z_iz_i^* \right)^{-1}z_j}{1+c_N\frac1Nz_j^*\left( \frac1n \sum_{\tau_i\geq m,i\neq j}  z_iz_i^* \right)^{-1}z_j}.
	\end{align}

	From Lemma~\ref{le:uniform_conv_sum_zi} in the Appendix, we then have for all large $n$ a.s.
	\begin{align*}
		\frac1Nz_j^*\left( \frac1n \sum_{\substack{i\neq j \\ \tau_i\geq m}}  z_iz_i^* \right)^{-1}z_j &= \frac1{\nu_n( [m,\infty) )} \frac1Nz_j^*\left( \frac1{|\{\tau_i\geq m\}|} \sum_{\substack{i\neq j \\ \tau_i\geq m}}  z_iz_i^* \right)^{-1}z_j \\
		&< \frac1{1-\varepsilon} \frac1{1-\frac{c_N}{1-\varepsilon}} = \frac1{1-c_N-\varepsilon}.
	\end{align*}
	Now, since $t\mapsto t/(1+c_Nt)$ is increasing, for all large $n$ a.s.
	\begin{align*}
		\frac{\frac1Nz_j^*\left( \frac1n \sum_{\substack{i\neq j \\ \tau_i\geq m}}  z_iz_i^* \right)^{-1}z_j}{1+c_N\frac1Nz_j^*\left( \frac1n \sum_{\substack{i\neq j \\ \tau_i\geq m}}  z_iz_i^* \right)^{-1}z_j}&< \frac{1}{1-c_N-\varepsilon}\frac1{1+c_N\frac1{1-c_N-\varepsilon}} = \frac1{1-\varepsilon}.
	\end{align*}
	As $\varepsilon < 1 - \phi_\infty^{-1}$, $(1-\varepsilon)^{-1}<\phi_\infty$ so that, from the inequality above, we can apply $\phi^{-1}$ on both sides of \eqref{eq:psi_ineq} to obtain, for all large $n$ a.s.
	\begin{align*}
		g^{-1}(md_{\max}) \leq \phi^{-1}\left( \frac1{1-\varepsilon}\right)
	\end{align*}
	hence
	\begin{align*}
		d_{\max} \leq \frac1m g\left( \phi^{-1}\left( \frac1{1-\varepsilon}\right) \right)
	\end{align*}
	from which $d_{\max}$ is uniformly bounded for all large $n$ a.s. Reverting all inequalities, we can proceed similarly with $d_{\min}$ by choosing $\varepsilon>0$ small and $M<\infty$ such that $\nu_n ( [0,M) )>1-\varepsilon$ for all large $n$ a.s. (which holds from the tightness of $\nu_n$). This shows that $d_{\min}$ is uniformly bounded away from zero and this completes the proof.
\end{proof}

Equipped with Lemmas~\ref{le:gamma_N}, \ref{le:convergence_gamma}, and \ref{le:di_bounded}, we are now in position to develop the core of the proof. For readability, we divide the proof in two parts. In the first part, we will assume that $\tau_1,\ldots,\tau_n$ have a uniformly bounded support. This will greatly simplify the calculus and will allow for a better understanding of the main arguments; in particular, the technical Assumption~\ref{ass:tau} will be irrelevant in this part. Then in a second part, we relax the boundedness assumption and fully exploit Assumption~\ref{ass:tau} in a more technical proof.

\subsubsection{Bounded $\tau_i$.}
\label{sec:bounded_tau}

	First assume $\tau_1,\ldots,\tau_n\leq M$ a.s. for some $M>0$. We follow here a similar path as in \citep{paper} but slightly more involved. Define
\begin{align}
	\label{eq:def_ei}
	e_i \triangleq \frac{v(\tau_i d_i)}{v(\tau_i \gamma_N)} > 0
\end{align}
with $\gamma_N$ any value given by Lemma~\ref{le:gamma_N} and with $d_i$ still defined as $d_i=\frac1Nz_i^*\hat{C}_{(i)}^{-1}z_i$. Up to labeling change, we reorder the $e_i$'s as $e_1\leq \ldots \leq e_n$. Our goal is to show that $e_1\asto 1$ and $e_n\asto 1$ (hence $\max_{1\leq i\leq n}|e_i-1|\asto 0$), which we will prove by a contradiction argument.

For any $j=1,\ldots,n$, we have
\begin{align}
	e_j &= \frac{v \left( \tau_j \frac1N z_j^* \hat{C}_{(j)}^{-1} z_j \right)}{v(\tau_j \gamma_N)} \nonumber \\
	&= \frac{v \left( \tau_j \frac1N z_j^* \left( \frac1n \sum_{i\neq j} \tau_i v(\tau_i d_i) z_iz_i^* \right)^{-1} z_j \right)}{v(\tau_j \gamma_N)} \nonumber \\
	\label{eq:bounded}
	&=\frac{v \left( \tau_j \frac1N z_j^* \left( \frac1n \sum_{i\neq j} \tau_i v(\tau_i \gamma_N) e_i z_iz_i^* \right)^{-1} z_j \right)}{v(\tau_j \gamma_N)} \\
	&\leq \frac{v \left( \tau_j \frac1N z_j^* \left( \frac1n \sum_{i\neq j} \tau_i v(\tau_i \gamma_N) e_n z_iz_i^* \right)^{-1} z_j \right)}{v(\tau_j \gamma_N)} \nonumber \\
	\label{eq:ineq_j}
	&= \frac{v \left( \frac{\tau_j}{e_n} \frac1N z_j^* \left( \frac1n \sum_{i\neq j} \tau_i v(\tau_i \gamma_N) z_iz_i^* \right)^{-1} z_j \right)}{v(\tau_j \gamma_N)}
\end{align}
where the inequality arises from $v$ being non-increasing and from \cite[Corollary~7.7.4]{HOR85}. Similarly, for each $j$,
\begin{align}
	e_j \geq \frac{v \left( \frac{\tau_j}{e_1} \frac1N z_j^* \left( \frac1n \sum_{i\neq j} \tau_i v(\tau_i \gamma_N) z_iz_i^* \right)^{-1} z_j \right)}{v(\tau_j \gamma_N)}.
\end{align}

From Lemma~\ref{le:convergence_gamma}, let now $0<\varepsilon_n<\gamma_N$, $\varepsilon_n\downarrow 0$, be such that, for all large $n$ a.s. and for all $j\leq n$,
\begin{align*}
	\gamma_N - \varepsilon_n < \frac1N z_j^* \left( \frac1n \sum_{i\neq j} \tau_i v(\tau_i \gamma_N) z_iz_i^* \right)^{-1} z_j < \gamma_N + \varepsilon_n.
\end{align*}

In particular, since $v$ is non-increasing, taking $j=n$ in \eqref{eq:ineq_j} and applying the left-hand inequality,
\begin{align*}
	e_n < \frac{v\left( e_n^{-1} \tau_n (\gamma_N - \varepsilon_n) \right)}{v(\tau_n \gamma_N)}
\end{align*}
or equivalently
\begin{align}
	\label{eq:inf0}
	\frac{e_n v(\tau_n \gamma_N)}{v\left( e_n^{-1} \tau_n (\gamma_N - \varepsilon_n) \right)} < 1.
\end{align}
By the definition of $\psi$, this can be further rewritten
\begin{align}
	\label{eq:inf1}
	\left(1-{\varepsilon_n}{\gamma_N}^{-1}\right) \frac{\psi(\tau_n \gamma_N)}{\psi\left( e_n^{-1} \tau_n \gamma_N (1-\varepsilon_n \gamma_N^{-1})\right)} < 1.
\end{align}

Assume now that, for some $\ell>0$, $e_n>1+\ell$ infinitely often and let us restrict the sequence $e_n$ to those indexes for which $e_n>1+\ell$. 

We distinguish two scenarios. First, assume that $\liminf_n \tau_{n}=0$. Then, by the definition \eqref{eq:def_ei} and since both $d_n$ and $\gamma_N$ are uniformly bounded (Lemma~\ref{le:gamma_N} and Lemma~\ref{le:di_bounded}), on some subsequence satisfying $\lim_n \tau_{n}=0$, $e_{n}\asto 1$, in contradiction with $e_{n}>1+\ell$.

We must then have $\liminf_n \tau_{n}>\tau_-$ for some $\tau_->0$ along with $\tau_n\leq M$ a.s. for some $M>0$ (bounded $\tau_i$ assumption). Then, since $\gamma_N$ is bounded and bounded away from zero for all large $n$ a.s., so is $\tau_n\gamma_N$. Considering a further subsequence over which $\tau_n\gamma_N\to x>0$ and $c_N\to c$, we then have, with $\psi_c(x)=\lim_{c_N\to c}\psi(x)$ (recall that $\psi$ depends on $c_N$ through $g$),
\begin{align}
	\lim_n \left(1-{\varepsilon_n}{\gamma_N}^{-1}\right) \frac{\psi(\tau_n \gamma_N)}{\psi\left( e_n^{-1} \tau_n \gamma_N (1-\varepsilon_n \gamma_N^{-1})\right)} \geq \frac{\psi_c(x)}{\psi_c( (1+\ell)^{-1}x)} > 1 \label{eq:fundamental_phi}
\end{align}
which contradicts \eqref{eq:inf1}. Gathering the results and reconsidering the initial sequence $e_n$ (i.e. not a subsequence) we then have, for each $\eta>0$, $e_n\leq 1+\ell$ for all large $n$ a.s.

Symmetrically, we obtain that, for some $\varepsilon_n\downarrow 0$ and for all large $n$ a.s.
\begin{align*}
	\frac{e_1 v(\tau_1 \gamma_N)}{v\left( e_1^{-1} \tau_1 (\gamma_N + \varepsilon_n) \right)} > 1.
\end{align*}
From this, we conclude similar to above that, for each $\ell>0$ small, $e_1\geq 1-\ell$, for all large $n$ a.s. so that, finally
\begin{align*}
	\max_{1\leq i\leq n} \left| e_i - 1 \right| \asto 0
\end{align*}
or, by uniform boundedness of the $\tau_i$ and $\gamma_N$,
\begin{align*}
	\max_{1\leq i\leq n} \left| v(\tau_id_i) - v(\tau_i \gamma_N) \right| \asto 0.
\end{align*}
Hence, letting $\ell>0$ and recalling that $\tau_iv(\tau_i\gamma_N)=\psi(\tau_i\gamma_N)/\gamma_N$, for all large $n$ a.s.
\begin{align}
	\label{eq:sandwich}
	(1-\ell) \frac1n \sum_{i=1}^n \frac{\psi(\tau_i\gamma_N)}{\gamma_N} z_iz_i^* \preceq \frac1n \sum_{i=1}^n v(\tau_i d_i) \tau_i z_iz_i^* \preceq (1+\ell) \frac1n \sum_{i=1}^n \frac{\psi(\tau_i\gamma_N)}{\gamma_N} z_iz_i^*.
\end{align}
Therefore, since $\gamma_N>\gamma_-$ and $\left\Vert\frac1n\sum_{i=1}^nz_iz_i^*\right\Vert < (1+\sqrt{c_+})^2$ for all large $n$ a.s. \citep{SIL98},
\begin{align*}
	\left\Vert \hat{C}_N - \hat{S}_N \right\Vert \leq 2\ell (1+\sqrt{c_+})^2 \frac{\psi_\infty}{\gamma_-}
\end{align*}
where $\hat{S}_N=\gamma_N^{-1} \frac1n \sum_{i=1}^n \psi(\tau_i \gamma_N) z_iz_i^*$. Since $\ell$ is arbitrary, the difference tends to zero a.s. as $n\to\infty$, which concludes the proof for $\tau_i<M$ a.s. and for $C_N=I_N$.

If $C_N\neq I_N$ is positive definite, remark simply that neither $d_i$ nor $\gamma_N$ are affected in their values, so that the effect of $C_N$ first appears in \eqref{eq:sandwich} with $z_i$ having $C_N\neq I_N$ as covariance matrix. But then, in this case, since $\Vert \frac1n\sum_{i=1}^n z_iz_i^* \Vert<(1+\sqrt{c_+})^2\limsup_N \Vert C_N\Vert<\infty$ (Assumption~\ref{ass:xi}), the last arguments still hold true and the result is also proved for these $C_N$.

\bigskip

Note the importance of the assumption on $\phi$ being increasing and not simply non-decreasing (as in \citep{MAR76}) to ensure that \eqref{eq:fundamental_phi} is a strict inequality. If this were to be replaced by ``$\geq 1$'', no contradiction with \eqref{eq:inf1} could be evoked. There does not seem to be any easy way to work this limitation around. Similar reasons explain why Tyler robust estimator discussed in Section~\ref{sec:conclusion} cannot be analyzed in the same way as Maronna estimator. All the same, when $\tau_1,\ldots,\tau_n$ have unbounded support with growing $n$, the left-hand side of \eqref{eq:fundamental_phi} may equal one provided $\limsup_n\tau_n=\infty$, which is not excluded. For this reason, a specific treatment is necessary where the set of $\{\tau_i\}_{i=1}^n$ is split into a large bounded set of $\tau_i$ and a small set of large $\tau_i$. This is the approach followed in the second part of the proof below.

\subsubsection{Unbounded $\tau_i$.}
\label{sec:unbounded_tau}

	We now relax the boundedness assumption on the support of the distribution of $\tau_1$ and use Assumption~\ref{ass:tau} instead.

	Since $\{\nu_n\}_{n=1}^\infty$ is tight, we can exhibit pairs $(\eta,M_\eta)$ with $\eta\downarrow 0$ as $M_\eta\uparrow \infty$ such that, for all large $n$ a.s. $\nu_n( (M_\eta,\infty) )<\eta$. Let us fix such a pair $(\eta,M_\eta)$ with $\eta$ small and restrict ourselves to a subsequence where $\nu_n( (M_\eta,\infty) )<\eta$ for all $n$. Denote $\mathcal C_\eta=\{i,\tau_i\leq M_\eta\}$ with cardinality $|\mathcal C_\eta|/n=1-\nu_n( (M_\eta,\infty) )$.

	We follow the same steps as in the previous proof but differentiating between indexes in $\mathcal C_\eta$ and indexes in $\mathcal C_\eta^c$. Also we denote
	\begin{align*}
		e^\eta_i \triangleq \frac{v(\tau_i d_i)}{v(\tau_i \gamma_N^\eta)}
	\end{align*}
	where $\gamma_N^\eta$ is the unique positive solution to the equation in $\gamma$
	\begin{align*}
		1 = \frac1n \sum_{i\in \mathcal C_\eta} \frac{\psi(\tau_i\gamma)}{1+c_N\psi(\tau_i \gamma)}.
	\end{align*}

	Recall first from Remark~\ref{rem:lemma1_eta} that the conclusions of Lemma~\ref{le:gamma_N} are still valid and importantly in what follows, that $\gamma_-<\gamma_N^\eta<\gamma_+$ for some $\gamma_-,\gamma_+>0$, for all large $N$ irrespective of $\eta<\eta_0$ for some $\eta_0$ small. This uniform control of $\gamma_N^\eta$ with respect to $\eta$ plays a key role here. For the moment, we do not make explicit the sufficiently small value of $\eta_0$ that is needed in the following; all what will matter if that we can always choose $\eta$ arbitrarily small from here.

	Let $j\in\mathcal C_\eta$ and denote $\psi_\infty$ any upper bound on $\psi^N_\infty$ for all $N$. Then, similar to \eqref{eq:bounded}, with $e^\eta_{\bar 1}=\min_{i\in\mathcal C_\eta} \{e^\eta_i\}$ and $e^\eta_{\bar n}=\max_{i\in\mathcal C_\eta} \{e^\eta_i\}$, 
	\begin{align*}
		e^\eta_j &= \frac{v\left( \tau_j \frac1N z_j^* \left( \frac1n \sum_{ \substack{i\in \mathcal C_\eta \\ i \neq j} } \tau_i v(\tau_i \gamma^\eta_N) e^\eta_i z_iz_i^* + \frac1n \sum_{ \substack{i\in \mathcal C_\eta^c} } \tau_i v(\tau_i d_i) z_iz_i^* \right)^{-1} z_j \right)}{v(\tau_j \gamma^\eta_N)} \\
		&\leq \frac{v\left( \tau_j \frac1N z_j^* \left( \frac1n \sum_{ \substack{i\in \mathcal C_\eta \\ i \neq j} } \tau_i v(\tau_i \gamma^\eta_N) e^\eta_{\bar n} z_iz_i^* + \frac1n \frac{\psi_\infty}{d_-}  \sum_{i \in \mathcal C_\eta^c} z_iz_i^* \right)^{-1}z_j \right)}{v(\tau_j \gamma^\eta_N)} \\
		&= \frac{v\left( \frac{\tau_j}{e^\eta_{\bar n}} \frac1N z_j^* \left( \frac1n \sum_{ \substack{i\in \mathcal C_\eta \\ i \neq j} } \tau_i v(\tau_i \gamma^\eta_N) z_iz_i^* + \frac1n \frac{\psi_\infty}{ e^\eta_{\bar n}d_-} \sum_{i \in \mathcal C_\eta^c} z_iz_i^* \right)^{-1}z_j \right)}{v(\tau_j \gamma^\eta_N)}
	\end{align*}
	where the first inequality uses $d_i>d_-$ for all large $n$ a.s (Lemma~\ref{le:di_bounded}). Since $e^\eta_{\bar n}=\frac{v(\tau_{\bar n}d_{\bar n})}{v(\tau_{\bar n}\gamma_N^\eta)} = \frac{\psi(\tau_{\bar n} d_{\bar n})}{\psi(\tau_{\bar n} \gamma^\eta_N)}\frac{\gamma^\eta_N}{d_{\bar n}}$, with the bounds derived previously (Remark~\ref{rem:lemma1_eta} and Lemma~\ref{le:di_bounded}), $e^\eta_{\bar n}$ is almost surely bounded and bounded away from zero for all large $n$ a.s., irrespective of $\eta$ small enough (if $\liminf_n \tau_{\bar n}=0$, the first equality ensures $\liminf_n e^\eta_{\bar n}>0$ while if $\limsup_n \tau_{\bar n}=\infty$, the second equality ensures $\limsup_n e^\eta_{\bar n}<\infty$). Thus, in particular, $e^\eta_{\bar n}>e_-$ for some $e_->0$ for all large $n$ a.s. From this observation, for all large $n$ a.s.
	\begin{align}
		e^\eta_j &\leq \frac{v\left( \frac{\tau_j}{e^\eta_{\bar n}} \frac1N z_j^* \left( \frac1n \sum_{ \substack{i\in \mathcal C_\eta \\ i \neq j} } \tau_i v(\tau_i \gamma^\eta_N) z_iz_i^* + \frac1n \frac{\psi_\infty}{d_-e_-} \sum_{i \in \mathcal C_\eta^c} z_iz_i^* \right)^{-1}z_j \right)}{v(\tau_j \gamma^\eta_N)} \nonumber \\
		\label{eq:partial_C}		&= \frac{v\left( \frac{\tau_j}{e^\eta_{\bar n}} \left[ \frac1N z_j^* \left( \frac1n \sum_{ \substack{i\in \mathcal C_\eta \\ i \neq j} } \tau_i v(\tau_i \gamma^\eta_N) z_iz_i^*\right)^{-1}z_j +  w_{j,n} \right] \right) }{v(\tau_j \gamma^\eta_N)}
	\end{align}
	where we defined
	\begin{align*}
		w_{j,n} &\triangleq \frac1N z_j^* \left( A_{\eta,(j)} + B_\eta \right)^{-1}z_j - \frac1N z_j^* A_{\eta,(j)}^{-1}z_j
	\end{align*}
	with
	\begin{align*}
		A_{\eta,(j)} \triangleq \frac1n \sum_{ \substack{i\in \mathcal C_\eta \\ i \neq j} } \tau_i v(\tau_i \gamma^\eta_N) z_iz_i^*,\quad B_\eta \triangleq \frac1n \frac{\psi_\infty}{d_-e_-} \sum_{i \in \mathcal C_\eta^c} z_iz_i^*.
	\end{align*}
	Note that $A_{\eta,(j)}^{-1}$ is well defined as $A_{\eta,(j)}$ is invertible for all large $n$ a.s. provided $\eta$ is small enough. Working on $w_{j,n}$, and using in particular $|x^*y|\leq \sqrt{x^*x}\sqrt{y^*y}$ for vectors $x,y$, we obtain 
	\begin{align*}
		|w_{j,n}| &= \left| \frac1N z_j^* \left( A_{\eta,(j)} + B_\eta \right)^{-1} B_\eta A_{\eta,(j)}^{-1} z_j \right| \\
		&\leq \sqrt{\frac1N z_j^* (A_{\eta,(j)}+B_\eta)^{-1} B_\eta (A_{\eta,(j)}+B_\eta)^{-1} z_j} \sqrt{\frac1N z_j^* A_{\eta,(j)}^{-1} B_\eta A_{\eta,(j)}^{-1} z_j}.
	\end{align*}
	Clearly $A_{\eta,(j)}+B_\eta\succeq A_{\eta,(j)}$ and, for some $\kappa>0$, and for all $j\in\mathcal C_\eta$, $\lambda_1(A_{\eta,(j)})>\kappa>0$ almost surely. Indeed, with the same derivation as \eqref{eq:ineq_md}, for any $m>0$ satisfying $\nu_n([m,M_\eta])>c_+$ for all $n$ a.s. (this may require $M_\eta$ large enough), $\lambda_1(A_{\eta,(j)}) \geq mv(m\gamma_+) \lambda_1( \frac1n \sum_{\tau_i \in [m,M_\eta],i\neq j} z_iz_i^*)$ away from zero for all large $n$ a.s., independently of $\eta$ small enough (Lemma~\ref{le:uniform_conv_sum_zi}). Now, with the same approach as in the proof of Lemma~\ref{le:convergence_gamma}, i.e. exploiting $z_j=\sqrt{\bar N}\tilde{z}_j/\Vert \tilde{y}_j\Vert$, Lemma~\ref{le:trace_lemma}, Lemma~\ref{le:rank1perturbation}, the union bound on the $j\in\mathcal C_\eta$, and \citep{SIL98} (the latter ensuring that $B_\eta$ has bounded spectral norm a.s.),
	\begin{align*}
		\max_{j\in \mathcal C_\eta}\left|\frac1Nz_j^* A_{\eta,(j)}^{-1} B_\eta A_{\eta,(j)}^{-1} z_j - \frac1N \tr A_\eta^{-1} B_\eta A_\eta^{-1} \right| \asto 0
	\end{align*}
	with $A_\eta=A_{\eta,(j)}+\frac1n\tau_jv(\tau_j\gamma_N^\eta)z_jz_j^*$. From $|\tr XY|\leq \Vert X\Vert \tr Y$ for positive definite $Y$, 
	\begin{align*}
		\frac1N\tr A_\eta^{-1}B_\eta A_\eta^{-1} \leq \Vert A_\eta^{-2} \Vert \frac1N\tr B_\eta \leq \frac1{\kappa^2} \frac{\psi_\infty}{d_-e_-}\frac1n\sum_{i\in\mathcal C_\eta^c} \frac{\Vert z_i\Vert^2}N
	\end{align*}
	where $\frac1n\sum_{i\in\mathcal C_\eta^c} \frac{\Vert z_i\Vert^2}N - \frac{|\mathcal C^c_\eta|}n \asto 0$ since $\max_{1\leq j\leq n} \left|\frac{\Vert z_j\Vert^2}N-1\right| \asto 0$. Recalling that $\frac{|\mathcal C^c_\eta|}n=\nu_n( (M_\eta,\infty) )$, we then have for all large $n$ a.s. 
	\begin{align*}
		\max_{j\in\mathcal C_\eta} \frac1Nz_j^* A_{\eta,(j)}^{-1} B_\eta A_{\eta,(j)}^{-1} z_j &< \frac{2\nu_n( (M_\eta,\infty) )}{\kappa^2}\frac{\psi_\infty}{d_-e_-}.
	\end{align*}
	The same reasoning holds for $\frac1Nz_j^* (A_{\eta,(j)}+B_\eta)^{-1} B_\eta (A_{\eta,(j)}+B_\eta)^{-1} z_j$. Finally, we conclude
	\begin{align}
		\label{eq:wn_bound}
		\max_{j\in \mathcal C_\eta}|w_{j,n}| &\leq K\nu_n( (M_\eta,\infty) )
	\end{align}
	for all large $n$ a.s. with $K>0$ constant, independent of $\eta$.  

	Now that $w_{j,n}$ is controlled for all $j\in\mathcal C_\eta$, we can proceed similar to the proof in the bounded $\tau_i$ case. First, for any fixed $\eta>0$ small enough, Remark~\ref{rem:lemma2_eta} ensures that there exists a sequence $\varepsilon^\eta_n\downarrow 0$, such that a.s.
	\begin{align}
		\label{eq:conv_gamma_C}
		\max_{j\in \mathcal C_\eta} \left| \frac1N z_j^* \left( \frac1n \sum_{ i\in \mathcal C_\eta,i \neq j } \tau_i v(\tau_i \gamma^\eta_N) z_iz_i^*\right)^{-1}z_j - \gamma_N^\eta\right| \leq \varepsilon^\eta_n.
	\end{align}
	Combining \eqref{eq:partial_C}, \eqref{eq:wn_bound}, and \eqref{eq:ej_bounded_v}, we then have for all large $n$ a.s. and for all $j\in\mathcal C_\eta$
	\begin{align}
		\label{eq:ej_bounded_v}
		e^\eta_j \leq \frac{v\left( \frac{\tau_j}{e^\eta_{\bar n}} \left( \gamma_N^\eta - \varepsilon^\eta_n - K \nu_n( (M_\eta,\infty) ) \right) \right)}{v(\tau_j \gamma^\eta_N)}
	\end{align}
	which, for $j=\bar n$, is
	\begin{align*}
		e^\eta_{\bar n} \leq \frac{v\left( \frac{\tau_{\bar n}}{e^\eta_{\bar n}} \left( \gamma_N^\eta - \varepsilon^\eta_n - K \nu_n( (M_\eta,\infty) ) \right) \right)}{v(\tau_{\bar n} \gamma^\eta_N)}.
	\end{align*}
	Using the definition of $\psi$, this reads equivalently
	\begin{align*}
		\left( 1 - \frac{\varepsilon^\eta_n + K\nu_n( (M_\eta,\infty) )}{\gamma_N^\eta} \right) \frac{\psi(\tau_{\bar n} \gamma_N^\eta)}{\psi\left( (e^\eta_{\bar n})^{-1} \tau_{\bar n} \gamma_N^\eta \left( 1 - \frac{\varepsilon^\eta_n + K\nu_n( (M_\eta,\infty) )}{\gamma_N^\eta} \right) \right)}<1
	\end{align*}
	which implies, from the growth of $\psi$,
	\begin{align*}
		\left( 1 - \frac{\varepsilon^\eta_n + K\nu_n( (M_\eta,\infty) )}{\gamma_N^\eta} \right) \frac{\psi(\tau_{\bar n} \gamma_N^\eta)}{\psi\left( (e^\eta_{\bar n})^{-1} \tau_{\bar n} \gamma_N^\eta\right)}<1.
	\end{align*}

	Adding $\frac{\varepsilon^\eta_n + K\nu_n( (M_\eta,\infty))}{\gamma_N^\eta}-1$ on both sides, this further reads
	\begin{align*}
		\left( 1 - \frac{\varepsilon^\eta_n + K\nu_n( (M_\eta,\infty))}{\gamma_N^\eta} \right) \frac{\psi(\tau_{\bar n} \gamma_N^\eta)-\psi\left( (e^\eta_{\bar n})^{-1} \tau_{\bar n} \gamma_N^\eta \right)}{\psi\left( (e^\eta_{\bar n})^{-1} \tau_{\bar n} \gamma_N^\eta \right)}< \frac{\varepsilon^\eta_n + K\nu_n( (M_\eta,\infty))}{\gamma_N^\eta}.
	\end{align*}
	or equivalently, if $\eta$ is taken small enough (recalling that $\gamma_N^\eta>\gamma_-$ uniformly on $\eta$ small),
	\begin{align}
		\label{eq:unbounded_fund_ineq}
		\frac{\psi(\tau_{\bar n} \gamma_N^\eta)-\psi\left( (e^\eta_{\bar n})^{-1} \tau_{\bar n} \gamma_N^\eta \right)}{\varepsilon^\eta_n + K\nu_n( (M_\eta,\infty))}< \frac{\psi\left( (e^\eta_{\bar n})^{-1} \tau_{\bar n} \gamma_N^\eta \right)}{\gamma_N^\eta\left( 1 - \frac{\varepsilon^\eta_n + K\nu_n( (M_\eta,\infty))}{\gamma_N^\eta} \right) } < \frac{2\psi_\infty}{\gamma_-}
	\end{align}
	where the right-most bound holds for all large $n$ a.s. provided $\eta$ is chosen small enough.

	
	Assume $\limsup_n e^\eta_{\bar n}>1+\ell$ for some $\ell>0$. Then one must have $\liminf_n \tau_{\bar n}>\tau_-$ for \eqref{eq:ej_bounded_v} to remain valid, with $\tau_->0$ independent of $\eta$ small since $\gamma_-<\gamma_N^\eta<\gamma_+$ for all $n$ large a.s., both bounds being independent of $\eta$. Since $\tau_{\bar n}\gamma_N^\eta$ belongs to $[\tau_-\gamma_-,M_\eta \gamma_+]$ for all large $N$ a.s., taking the limit of \eqref{eq:unbounded_fund_ineq} over some converging subsequence over which $\tau_{\bar n}\gamma_N^\eta\to x^\eta\in [\tau_-\gamma_-,M_\eta \gamma_+]$, $c_N\to c$, and $\nu_n( (M_\eta,\infty))$ converges, ensures that
	\begin{align}
		\label{eq:min_x_C}
		\frac{\psi_c(x^\eta)-\psi_c\left( \frac1{1+\ell} x^\eta \right)}{\lim_n \nu_n( (M_\eta,\infty))}\leq K'
	\end{align}
	for $K'>0$ independent of $\eta$, with $\psi_c=\lim_{c_N\to c}\psi$.

	We now operate on $\eta$. If $\limsup_{\eta\to 0} x^\eta < \infty$, the left-hand side in \eqref{eq:min_x_C} diverges to $\infty$ as $\eta\to 0$ so that, starting with an $\eta$ sufficiently small and taking the limit over $n$ on the subsequence under consideration raises a contradiction. If instead $\limsup_{\eta\to 0} x^\eta = \infty$, then, since $x^\eta\leq M_\eta\gamma_+$,
	\begin{align*}
		\frac{\psi_c(x^\eta)-\psi_c\left( \frac1{1+\ell} x^\eta \right)}{\lim_n \nu_n( (M_\eta,\infty) )} \geq \frac{\psi_c(x^\eta)-\psi_c\left( \frac1{1+\ell} x^\eta \right)}{\lim_n\nu_n( (\frac{x^\eta}{\gamma_+},\infty) )}.
	\end{align*}
	Call $y^\eta=g^{-1}(x^\eta)$. Then, after some calculus,
	\begin{align*}
		\frac{\psi_c(x^\eta)-\psi_c\left( \frac1{1+\ell} x^\eta \right)}{\lim_n\nu_n ( (\frac{x^\eta}{\gamma_+},\infty) )} &= \frac{\phi(y^\eta)-\phi\left(\frac1{1+\ell}y^\eta\right)}{(1-c\phi(y^\eta))(1-c\phi(\frac{y^\eta}{1+\ell}))\lim_n \nu_n( (\frac{y^\eta}{\gamma_+(1-c\phi(y^\eta))},\infty) )} \\
		&> \frac{\phi(y^\eta)-\phi\left(\frac1{1+\ell}y^\eta\right)}{(1-c_+\phi_\infty)^2 \lim_n \nu_n( (\frac{y^\eta}{\gamma_+},\infty) )}.
	\end{align*}
	Since $y^\eta\to\infty$ as $x^\eta\to\infty$, from Assumption~\ref{ass:tau}, the right-hand side must go to $\infty$ as $x^\eta\to\infty$, or equivalently as $\eta\to 0$. Therefore, taking $\eta$ sufficiently small from the beginning and then bringing $n$ large on the subsequence under study leads to a contradiction. Consequently, we must have $\limsup_n e^\eta_{\bar n}\leq 1+\ell$ a.s. A similar reasoning shows that $\liminf_n e^\eta_{\bar 1}\geq 1-\ell$ a.s., for any given $\ell>0$. We conclude that
	\begin{align*}
		\max_{j\in\mathcal C_\eta} \left| e^\eta_j - 1 \right| \asto 0.
	\end{align*}
	
	We now have to deal with $e^\eta_j$ for $j\in \mathcal C_\eta^c$. For such a $j$,
\begin{align*}
	d_j = \frac1N z_j^* \left( \frac1n \sum_{i\in\mathcal C_\eta} \tau_i v(\tau_i \gamma^\eta_N) e^\eta_i z_iz_i^* + \frac1n \sum_{i\in\mathcal C_\eta^c,i\neq j} \frac{\psi(\tau_i d_i)}{d_i} z_iz_i^*  \right)^{-1}z_j.
\end{align*}
But then, from the same reasoning as with the $w_{j,n}$ above, we have that
\begin{align*}
	\max_{j\in\mathcal C_\eta^c} \left| d_j - \frac1N z_j^* (A_\eta')^{-1}z_j \right|
	&= \max_{j\in\mathcal C_\eta^c} \left| \frac1Nz_j^* (A_\eta'+B_{\eta,(j)}')^{-1}B_{\eta,(j)}'(A_\eta')^{-1}z_j \right|
\end{align*}
with 
\begin{align*}
	A_\eta' = \frac1n \sum_{i\in\mathcal C_\eta} \tau_i v(\tau_i \gamma^\eta_N) e^\eta_i z_iz_i^*, \quad B_{\eta,(j)}' = \frac1n \sum_{i\in \mathcal C_\eta^c,i\neq j} \frac{\psi(\tau_i d_i)}{d_i} z_iz_i^*
\end{align*}
which is easily bounded (using the fact that both $\psi(\tau_i d_i)/d_i$, $i\in\{1,\ldots,n\}$, and $e_i^\eta$, $i\in\mathcal C_\eta$, are bounded and bounded away from zero for all large $n$ a.s., Lemma~\ref{le:di_bounded}) as
\begin{align*}
	\max_{j\in\mathcal C_\eta^c} \left| d_j - \frac1N z_j^* (A_\eta')^{-1}z_j \right| &< K\nu_n( (M_\eta,\infty) ) < K\eta
\end{align*}
for some $K>0$ for all large $n$ a.s. (see reasoning leading to \eqref{eq:wn_bound}).

Moreover, since $\max_{i\in\mathcal C_\eta}|e^\eta_i-1|\asto 0$, using $\tau_iv(\tau_i\gamma_N^\eta)\leq \psi_\infty\gamma_-^{-1}$ for all large $n$ a.s. and $\Vert \frac1n\sum_{i\in \mathcal C_\eta}z_iz_i^*\Vert$ a.s. bounded,
\begin{align*}
	\max_{j\in\mathcal C_\eta^c} \left| d_j - \frac1N z_j^* \left( \frac1n \sum_{i\in\mathcal C_\eta} \tau_i v(\tau_i \gamma^\eta_N) z_iz_i^* \right)^{-1}z_j \right| &< 2K\eta
\end{align*}
which further implies from Remark~\ref{rem:lemma2_eta} that for all large $n$ a.s. and for all $j\in\mathcal C_\eta^c$,
\begin{align*}
	\gamma_N^\eta - 2K\eta \leq d_j \leq \gamma_N^\eta + 2K\eta.
\end{align*}
Using the definition $e^\eta_j=\frac{\psi(\tau_jd_j)}{\psi(\tau_j\gamma_N^\eta)}\frac{\gamma_N^\eta}{d_j}$, the uniform bounds on $\gamma_n^\eta$, and the continuous growth of $\psi$ shows finally that, a.s.
\begin{align*}
	\limsup_n \max_{j\in\mathcal C_\eta^c} \left\{ \left| e^\eta_j - 1 \right| \right\} \leq \eta'
\end{align*}
for some $\eta'>0$ with $\eta'\to 0$ as $\eta\to 0$.

Gathering the results for $j\in\mathcal C_\eta$ and $j\in\mathcal C_\eta^c$, we therefore conclude that, for each $\ell>0$, there exists $\eta>0$ small enough such that a.s.
\begin{align*}
	1-\ell < \liminf_n \min_{1\leq i\leq n} e^\eta_i \leq \limsup_n \max_{1\leq i\leq n} e^\eta_i < 1+\ell.
\end{align*}

For such $\eta$ small, we then have, by definition of $e_i^\eta$ and from $\tau_iv(\tau_i\gamma_N^\eta)=\psi(\tau_i\gamma_N^\eta)/\gamma_N^\eta$,
	\begin{align}
		\label{eq:gamma_eta_conc}
		(1-\ell) \frac1n \sum_{i=1}^n \frac{\psi(\tau_i \gamma_N^\eta)}{\gamma_N^\eta} z_iz_i^* \preceq \frac1n \sum_{i=1}^n v(\tau_i d_i) \tau_i z_iz_i^* \preceq (1+\ell) \frac1n \sum_{i=1}^n \frac{\psi(\tau_i \gamma_N^\eta)}{\gamma_N^\eta} z_iz_i^*.
	\end{align}

It now remains to show that, for each $\varepsilon>0$, there exists $\eta>0$ for which $|\gamma_N^\eta - \gamma_N| < \varepsilon$ for all $n$ large a.s.
For this, observe that, by definition of $\gamma_N$ and $\gamma_N^\eta$,
	\begin{align*}
		1 = \frac1n \sum_{i\in\mathcal C_\eta} \frac{\psi(\tau_i \gamma^\eta_N)}{1+c_N \psi(\tau_i \gamma^\eta_N)} = \frac1n \sum_{i=1}^n \frac{\psi(\tau_i \gamma_N)}{1+c_N \psi(\tau_i \gamma_N)}
	\end{align*}
	so that, since $\psi/(1+c_N\psi)$ is increasing, we obtain $\gamma_N\leq \gamma_N^\eta$ and 
	\begin{align*}
		\frac1n \sum_{i\in \mathcal C_\eta^c} \frac{\psi(\tau_i\gamma_N)}{1+c_N\psi(\tau_i\gamma_N)} &= \frac1n\sum_{i\in\mathcal C_\eta} \frac{\psi(\tau_i \gamma^\eta_N)-\psi(\tau_i \gamma_N)}{(1+c_N\psi(\tau_i\gamma_N))(1+c_N\psi(\tau_i\gamma_N^\eta))} \geq 0.
	\end{align*}
	Take an interval $[m,M]$, $M<M_\eta$ (chosen once for all, independently of $M_\eta$ large), with $\nu_n([m,M])>\kappa>0$ for all large $n$ a.s. (possible from Assumption~\ref{ass:xi}--\ref{item:0}). Then we can further write
	\begin{align*}
		\frac1n \sum_{i\in \mathcal C_\eta^c} \frac{\psi(\tau_i\gamma_N)}{1+c_N\psi(\tau_i\gamma_N)} &\geq \frac1{(1+c_+\psi_\infty)^2}\frac1n\sum_{\tau_i \in [m,M]} \left(\psi(\tau_i \gamma^\eta_N)-\psi(\tau_i \gamma_N) \right) \\
		&\geq \frac{\kappa}{2(1+c_+\psi_\infty)^2} \min_{x\in [m,M]}\left(\psi( x \gamma^\eta_N)-\psi( x \gamma_N) \right)
	\end{align*}
	with the second inequality valid for all large $n$ a.s. Now, for sufficiently small $\eta$, the left-hand side can be made arbitrarily small. Since $\gamma_N$ and $\gamma_N^\eta$ are uniformly bounded and bounded away from zero (irrespective of $\eta$ small), if $|\gamma^\eta_N - \gamma_N|$ were uniformly away from zero for all $\eta$ small, so would be the right-hand side, which is in contradiction with our previous statement. Therefore, for each $\varepsilon>0$, one can choose $\eta$ so that $|\gamma_N-\gamma_N^\eta|<\varepsilon$ for all $n$ large a.s.

	Now, by uniform continuity of $\psi$ on bounded intervals along with the fact that $\psi(x)\uparrow \psi_\infty$, from \eqref{eq:gamma_eta_conc}, taking $\eta$ small enough, for all large $n$ a.s.
	\begin{align}
		(1-\ell)^2 \frac1n \sum_{i=1}^n \frac{\psi\left( \tau_i \gamma_N \right)}{\gamma_N} z_iz_i^* \preceq \hat{C}_N \preceq (1+\ell)^2 \frac1n \sum_{i=1}^n \frac{\psi\left( \tau_i \gamma_N \right)}{\gamma_N} z_iz_i^*
	\end{align}
	which therefore implies, with the same arguments as in the case $\tau_i$ bounded, that $\Vert \hat{C}_N-\hat{S}_N\Vert\asto 0$, when $C_N=I_N$. The arguments of the case $\tau_i$ bounded still hold for $C_N\neq I_N$ satisfying Assumption~\ref{ass:xi}-\ref{item:CN}). This completes the proof.

\section{Conclusion}
\label{sec:conclusion}
This article introduces a large dimensional analysis for robust estimators of scatter matrices of the Maronna-type from elliptically distributed samples. We specifically showed that, under mild assumptions, the Maronna estimator behaves similar to a classical sample covariance matrix model as both the population and sample sizes grow large. This study opens new roads in the analysis of signal processing methods based on robust scatter matrix estimation. In a similar manner as in \cite[Theorem~6]{MAR76}, it is believed that second order statistics for well behaved functionals of $\hat{C}_N$ can be further analyzed, which would provide more information on the asymptotic fluctuations of $\hat{C}_N-\hat{S}_N$. The mathematical treatment developed in the proofs of our present results however shows some strong limitations for hypothetical extensions to other robust scatter matrix estimates. In particular, the important Tyler robust estimator \citep{TYL87,PAS08}, given by the unique solution (up to a scale factor) to \eqref{eq:def_hatCN} for $u(x)=1/x$, cannot be analyzed from the present method which relies essentially on $\phi(x)=xu(x)$ being increasing. Although extensive simulations suggest that similar conclusions hold for Tyler estimator, there is to this day no approach to tackle this problem. 

\appendix

\section{Some Lemmas}

\begin{lemma}[Rank-one perturbation] \cite[Lemma~2.6]{SIL95}
	\label{le:rank1perturbation}
	Let $v\in\CC^N$, $A,B\in\CC^{N\times N}$ nonnegative definite, and $x>0$. Then
	\begin{align*}
		\left| \tr B\left(A+vv^*+xI_N\right)^{-1} - \tr B \left(A+xI_N\right)^{-1} \right| \leq x^{-1}\Vert B\Vert.
	\end{align*}
\end{lemma}

\begin{lemma}[Trace lemma] \cite[Lemma~B.26]{SIL06}
	\label{le:trace_lemma}
	Let $A\in\CC^{N\times N}$ be non-random and $y=[y_1,\ldots,y_N]^\trans\in\CC^N$ be a vector of independent entries with $\EE [y_i]=0$, $\EE[|y_i|^2]=1$, and $\EE[|y_i|^\ell]\leq \zeta_\ell$ for all $\ell\leq 2p$, with $p\geq 2$. Then,
	\begin{align*}
		\EE\left[\left| y^* Ay - \tr A \right|^p\right]\leq C_p \left( (\zeta_4 \tr AA^*)^{\frac{p}2} + \zeta_{2p} \tr(AA^*)^{\frac{p}2} \right)
	\end{align*}
	for $C_p$ a constant depending on $p$ only.
\end{lemma}

\begin{lemma}
	\label{le:uniform_conv_sum_zi}
	Let $z_1,\ldots,z_n\in\CC^N$ be independent unitarily invariant vectors with $\Vert z_i\Vert^2=N$. Then, if $0<\liminf_n N/n \leq \limsup_n N/n<1$,
	\begin{align*}
		\max_{1\leq j\leq n} \left| \frac1Nz_j^*\left( \frac1n\sum_{i=1}^n z_iz_i^* \right)^{-1}z_j - 1 \right| \asto 0.
	\end{align*}
	Moreover, there exists $\varepsilon>0$ such that, for all large $n$ a.s.
	\begin{align*}
	\lambda_1\left( \frac1n\sum_{i=1}^n z_iz_i^* \right) \geq \min_{1\leq j\leq n} \left\{ \lambda_1\left( \frac1n\sum_{ \substack{1\leq i\leq n \\ i\neq j} } z_iz_i^* \right) \right\} > \varepsilon.
	\end{align*}
\end{lemma}
\begin{proof}
	For readability, we denote $F=\frac1n\sum_{i=1}^nz_iz_i^*$, $F_{(j)}=F-\frac1nz_jz_j^*$, $\tilde{F}=\frac1n\sum_{i=1}^n\tilde{z}_i\tilde{z}_i^*$, and $\tilde{F}_{(j)}=\tilde{F}-\frac1n\tilde{z}_j\tilde{z}_j^*$, where we recall the relation $z_i=\sqrt{\bar N}\tilde{z}_i/\Vert \tilde{y}_i\Vert$ for $\tilde{z}_i$ zero mean $I_N$-covariance Gaussian and $\tilde{y}_i$ zero mean $I_{\bar N}$-covariance Gaussian (non-independent).
	From the proof of \cite[Lemma~2]{paper}, we have:
	\begin{align*}
		\max_{1\leq j\leq n} \left| \frac1N\tilde{z}_j^*\tilde{F}^{-1}\tilde{z}_j - 1 \right| \asto 0
	\end{align*}
	and, there exists $\varepsilon>0$ such that, for all large $n$ a.s.
	\begin{align}
		\label{eq:lambda1_tildeF}
		\lambda_1(\tilde{F}) \geq \min_{1\leq j\leq n}  \lambda_1 (\tilde{F}_{(j)}) > \varepsilon.
	\end{align}
	Now, 
	\begin{align*}
		\min_{1\leq j\leq n} \lambda_1 (F_{(j)}) &\geq \frac{\min_{1\leq j\leq n} \lambda_1(\tilde{F}_{(j)}) }{\max_{1\leq j\leq n} {\bar N}^{-1}\Vert \tilde{y}_j\Vert^2}
	\end{align*}
	Since $\max_{1\leq j\leq n} N^{-1}\Vert \tilde{z}_j\Vert^2\asto 1$ a.s. from standard probability results, we have that for all large $n$ a.s.
	\begin{align*}
		\lambda_1\left( F \right) \geq \min_{1\leq j\leq n} \lambda_1(F_{(j)}) > \varepsilon/2
	\end{align*}
	which already gives the second part of the lemma.
	Using only the outer inequality of \eqref{eq:lambda1_tildeF}, we now have, for all large $n$ a.s.
	\begin{align*}
		\max_{1\leq j\leq n} \left| \frac1N\tilde{z}_j^*F^{-1}\tilde{z}_j - \frac1N\tilde{z}_j^*\tilde{F}^{-1}\tilde{z}_j \right|
		&= \max_{1\leq j\leq n} \left| \frac1N\tilde{z}_j^*F^{-1}\left( \tilde{F}-F \right)\tilde{F}^{-1}\tilde{z}_j \right| \\
		&\leq \max_{1\leq j\leq n} \left\{ \frac1n\sum_{k=1}^n \left| 1 - \frac{\bar N}{\Vert \tilde y_k\Vert^2} \right| \left| \frac1N\tilde{z}_j^* \tilde{F}^{-1} \tilde{z}_k\right|^2\right\} \\
		&\leq \max_{1\leq k\leq n} \left| 1 - \frac{\bar N}{\Vert \tilde y_k\Vert^2} \right| \frac1N \left(\max_{1\leq k\leq n} \left\Vert \tilde{z}_k\right\Vert\right)^2  \frac{4}{\varepsilon^2} \\
		&\asto 0.
	\end{align*}
	Finally, for all large $n$ a.s.
	\begin{align*}
		\max_{1\leq j\leq n} \left| \frac1N\tilde{z}_j^*F^{-1}\tilde{z}_j - \frac1Nz_j^*F^{-1}z_j \right|&= \max_{1\leq j\leq n} \left\{\left| \frac1N\tilde{z}_j^*F^{-1}\tilde{z}_j \right| \left| 1 - \frac{\bar N}{\Vert \tilde y_k\Vert^2} \right| \right\}\\
		&\leq \frac2{\varepsilon} \max_{1\leq k\leq n} \left| 1 - \frac{\bar N}{\Vert \tilde y_k\Vert^2} \right| \max_{1\leq j\leq n} \frac1N \left\Vert \tilde{z}_j\right\Vert^2 \\
		&\asto 0.
	\end{align*}
	The proof is concluded by putting these results together.
\end{proof}

\bibliographystyle{elsarticle-harv}
\bibliography{/home/romano/Documents/PhD/phd-group/papers/rcouillet/tutorial_RMT/book_final/IEEEabrv.bib,/home/romano/Documents/PhD/phd-group/papers/rcouillet/tutorial_RMT/book_final/IEEEconf.bib,/home/romano/Documents/PhD/phd-group/papers/rcouillet/tutorial_RMT/book_final/tutorial_RMT.bib,robust_est}

\end{document}